\documentclass{article}
\usepackage[latin1]{inputenc}
\usepackage{amsfonts}
\usepackage{amsmath}
\usepackage{amssymb}
\usepackage{latexsym}
\usepackage[dvips]{graphics,color}

\newenvironment{proof}[1][Proof]{\noindent\textbf{#1.} }{$\hfill\Box$}

\numberwithin{equation}{section}

\begin{document}

\title{Sufficient Conditions for Existence of Solutions to Vectorial Differential
Inclusions and Applications}
\author{Ana Cristina Barroso
 \thanks{A.C. Barroso is with the Maths Dep. and CMAF, University of Lisbon,
Lisbon, Portugal.
E-mail: abarroso@ptmat.fc.ul.pt} \;
Gisella Croce
\thanks{G. Croce  is with the LMAH, Le Havre University, Le Havre,
France. E-mail: gisella.croce@univ-lehavre.fr} \;
 Ana Margarida Ribeiro
\thanks{A.M. Ribeiro is with the Maths Dep. and CMA, FCT-Universidade Nova de Lisboa,
Caparica, Portugal.
E-mail: amfr@fct.unl.pt}
}

\date{}

\maketitle

\begin{abstract}
In this paper we discuss the existence of solutions to vectorial
differential inclusions,
refining a result proved in Dacorogna and Marcellini
\cite{Dac-Marc-CR}. We
investigate sufficient conditions for existence, more flexible than those available in the literature, so
that important
applications can be fitted in the theory.
We also study some of these applications.
\end{abstract}

% place your newtheorem declarations here, like

\newtheorem{theorem}{Theorem}[section]
\newtheorem{lemma}[theorem]{Lemma}
\newtheorem{proposition}[theorem]{Proposition}
\newtheorem{remark}[theorem]{Remark}
\newtheorem{corollary}[theorem]{Corollary}
\newtheorem{criterion}[theorem]{Criterion}
\newtheorem{definition}[theorem]{Definition}
\newtheorem{example}[theorem]{Example}
\newtheorem{notation}[theorem]{Notation}
\newtheorem{summary}[theorem]{Summary}

\section{Introduction}

In this paper we discuss the existence of
$W^{1,\infty}(\Omega,\mathbb{R}^{N})$ solutions to the
vectorial differential inclusion problem
\begin{equation}\label{general differential inclusion}
\left\{\begin{array}{l}
Du(x)\in E,\ a.e.\ x\in\Omega,\vspace{0.2cm}\\
u(x)=\varphi(x),\ x\in\partial\Omega,
\end{array}\right.
\end{equation}
where $\Omega$ is an open subset of $\mathbb{R}^n$, $E$ is a given
subset of $\mathbb{R}^{N\times n}$ and
$\varphi:\overline{\Omega}\to\mathbb{R}^{N}$.
This problem has been intensively studied
by Dacorogna and Marcellini \cite{Dac-Marc-CR}, \cite{Dac-Marc}
through the Baire categories method (see also M\"uller and {\v{S}}ver{\'a}k
\cite{Muller-Sverakcounterexamples}).
Their result provides a sufficient condition for existence of solutions
related with the gradient of the boundary data. It asserts that, if this
gradient belongs to a convenient set enjoying the so called relaxation property with respect to the set $E$
(see Definition \ref{prop-relax}, and Theorem \ref{abstract existence theorem} due to
Dacorogna and Pisante \cite{Dacorogna-Pisante}) a dense set of solutions to
(\ref{general differential inclusion}) exists.

In the applications, direct verification of the relaxation property is a hard task and sufficient
conditions for it were also
obtained by Dacorogna and Marcellini \cite{Dac-Marc}, namely, the approximation property, cf.
Definition \ref{Approximation property} and Theorem~\ref{Relaxation property theorem for RcoE}
(see also \cite[Theorem 6.15]{Dac-Marc} for a more general version).
If such a property is satisfied, we can get as a sufficient condition for existence of solutions to
(\ref{general differential inclusion})
\begin{equation}\label{generalsufficientcondition}
D\varphi(x)\in E\cup \operatorname*{int}\operatorname*{Rco}E,\ a.e.\ x\in\Omega,
\end{equation}
where $\operatorname*{int}\operatorname*{Rco}E$ denotes the interior of the rank one convex hull
of the set $E$, that is,
the interior of the smallest rank one convex set containing $E$
(see Definition \ref{convexhulls}).

However, there are interesting applications for which $\operatorname*{int}\operatorname*{Rco}E$ is empty,
and thus, condition (\ref{generalsufficientcondition}) is much too restrictive. This is the case of a
well known problem solved by Kirchheim in \cite{Kirchheim} that we will discuss in Section \ref{sectionexamples}.
In view of this, our goal in Section \ref{sectionsufficientconditions} is to obtain sufficient
conditions for the relaxation property which can be handled in the applications and which are more
flexible than (\ref{generalsufficientcondition}). More precisely, we shall be able to deal with
subsets of the hull $\operatorname*{Rco}_fE$ defined as
$$\operatorname*{Rco}\nolimits_{f}E  =\left\{
\xi\in\mathbb{R}^{N\times n}:f\left( \xi\right)  \leq0\text{, for
every rank one convex }
f\in\mathcal{F}^{E}\right\},$$
where
$\mathcal{F}^{E}=\{f:\mathbb{R}^{N\times n}\rightarrow \mathbb{R}:\left.  f\right\vert _{E}\leq0\}$,
and which is, in general, a larger hull than $\operatorname*{Rco}E$.
Recovering results due to Pedregal \cite{Pedregal} and to M{\"u}ller and {\v{S}}ver{\'a}k
\cite{Muller-Sverakcounterexamples}, we obtain in
Theorem \ref{Rcofcharact} the following characterization of this type of hulls for compact sets $E$:
\begin{equation}\label{Rcofcharactintro}
\mathrm{Rco}_f E=\left\{\begin{array}{l}
\xi\in\mathbb{R}^{N\times n}:\ \forall\ \varepsilon>0\ \exists\ I\in\mathbb{N},\ \exists\
 (\lambda_i,\xi_i)_{i=1,...,I} \text{ with } \lambda_i>0,\vspace{0.2cm}\\  \displaystyle{
 \sum_{i=1}^I\lambda_i=1} \text{ satisfying } (H_I(U)),\ \xi=\sum_{i=1}^I\lambda_i\xi_i,\ \sum_{\substack{i=1\\\xi_i\notin
B_\varepsilon(E)}}^I\lambda_i<\varepsilon\end{array} \right\},
\end{equation}
where $U$ is an open and bounded set containing $\mathrm{Rco}_f E$ and the property
$(H_I(U))$ is introduced in Definition \ref{defH_I}.
Thanks to this characterization we will  prove, in particular, the following result
(cf. Corollary \ref{corol1}).

\begin{theorem}
Let $E\subset\mathbb{R}^{N\times n}$ be bounded and such that $E$ and
$\operatorname*{Rco}\nolimits_{f}E$ have the
approximation property with $K_\delta=\operatorname*{Rco}\nolimits_{f}E_\delta$ for some compact sets
$E_\delta\subset\mathbb{R}^{N\times n}$. Then $\mathrm{int}\operatorname*{Rco}\nolimits_{f}E$
has the relaxation property with respect to $E$.
\end{theorem}

From this theorem and from the Baire categories method it follows that, to ensure existence of solutions to
(\ref{general differential inclusion}) under the approximation property assumption,
condition (\ref{generalsufficientcondition}) can be replaced by
\begin{equation}\label{moregeneralsufficientcondition}
D\varphi(x)\in E\cup \operatorname*{int}\operatorname*{Rco}\nolimits_{f}E,\ a.e.\ x\in\Omega.
\end{equation}
Based on the characterization of the elements of $\operatorname*{Rco}\nolimits_{f}E$ given in
(\ref{Rcofcharactintro}) we will prove in
Theorem \ref{sufficientconditionforrelaxation} a more general sufficient condition for the
relaxation property which allows us to work with subsets of $\operatorname*{Rco}\nolimits_{f}E$.
This is very useful in the applications because many times the entire hull
$\operatorname*{Rco}\nolimits_{f}E$ is not known.
Moreover, characterizing $\operatorname*{Rco}\nolimits_{f}E$ (or $\operatorname*{Rco}E$) may lead to
complicated formulas and thus checking condition (\ref{moregeneralsufficientcondition})
(or (\ref{generalsufficientcondition})) becomes very difficult.
However, many problems can still be solved provided it is possible to work with convenient
subsets of $\operatorname*{Rco}\nolimits_{f}E$. This will be the case in the several applications
given in Section \ref{sectionexamples}.

We start, in Section \ref{KirchheimExample}, with the problem already solved by Kirchheim
\cite{Kirchheim} on the existence
of a non affine map with a finite number of gradients whose values are not rank one connected
and with an affine boundary condition. We will show that this example is still in the setting of the
Baire categories method thanks to the sufficient
conditions for the relaxation property proved in Section \ref{sectionsufficientconditions}.
In this case the set $E$ is a
finite set of matrices with no rank one connections. We observe that we don't need  to compute
$\operatorname*{Rco}\nolimits_{f}E$ and that we get existence of solutions whenever the affine
boundary data $\varphi$ satisfies $D\varphi\in K$, for a certain set
$K\subset \operatorname*{Rco}\nolimits_{f}E$.

Then we will come back to the problem, already considered by Croce \cite{Croce}, of
arbitrary compact isotropic subsets
$E$ of $\mathbb{R}^{2\times 2}$. Once again, our theory shows here its versatility. For
this type of sets the hull $\operatorname*{Rco}\nolimits_{f}E$ was characterized by Cardaliaguet and Tahraoui
\cite{CT}. Although it was proved in \cite{Croce} that this hull coincides with $\operatorname*{Rco}E$,
we are now able to apply the Baire categories method without using this information.

Finally we consider, in Section \ref{Differential inclusions for some SO(n) invariant sets}, the case
of sets $E$ for which a
constraint on the sign of the determinant is imposed on a set of isotropic matrices:
\begin{equation}\label{isotropic with restriction intro}
E=\left\{\xi \in \mathbb{R}^{n\times n}: (\lambda_1(\xi),\cdots, \lambda_n(\xi))\in \Lambda_E,\
\det\xi> 0\right\},
\end{equation}
where $\Lambda_E$ is a set contained in
$\{(x_1, \cdots,x_n)\in \mathbb{R}^n: 0< x_1 \leq \cdots \leq x_n\}$ and
$0 \leq \lambda_1(\xi) \leq \cdots \leq \lambda_n(\xi)$ are the singular values
of the matrix $\xi$ (cf. Section \ref{sectionexamples}).
Characterizing the hulls of such sets is quite
complicated and the only results available were obtained by Cardaliaguet and Tahraoui \cite{CT2}
in dimension $n=2$.
Considering a particular class of sets $E$ we will prove the following result
(cf. Theorem \ref{existence isotropic with determinant constraint two points}).

\begin{theorem}
Let
$
E=\{\xi \in \mathbb{R}^{2\times 2}: (\lambda_1(\xi),\lambda_2(\xi))\in \{(a_1,b_1),(a_2,b_2)\},\ \det \xi> 0\}
$
with $0<a_1<b_1<a_2<b_2$.
Let $\Omega\subset \mathbb{R}^2$ be a bounded open set and let
$\varphi \in C^1_{piec}(\overline{\Omega}, \mathbb{R}^2)$ be such that
$D \varphi \in E \cup \mathrm{int}\operatorname*{Rco}_fE$ a.e. in $\Omega$. Then there exists
a map $u\in \varphi+ W^{1,\infty}_0(\Omega, \mathbb{R}^2)$ such that
$Du(x) \in E$ for a.e.  $x$ in $\Omega$.
\end{theorem}

In addition, we are also able to establish sufficient conditions for sets $E$ of the form
(\ref{isotropic with restriction intro}) in dimension $n=2$ and $n=3$, working with a subset of
$\operatorname*{Rco}_fE$. In particular, in dimension 2, we prove the following result.

\begin{theorem}
Let $\Lambda_E$ be a subset of $\mathbb{R}^2$ containing the line segment joining two
distinct points $(a_1,a_2)$ and $(b_1,b_2)$ such that $0<a_1\le a_2$, $0<b_1\le b_2$, $a_1<b_1$,
$a_2<b_2$, and either $a_1<a_2$ or $b_1<b_2$. Let
$$E=\left\{\xi \in \mathbb{R}^{2\times 2}: (\lambda_1(\xi),\lambda_2(\xi))\in \Lambda_E,\
\det\xi> 0\right\},$$
$\Omega\subset \mathbb{R}^2$ be a bounded open set and
$\varphi \in C^1_{piec}(\overline{\Omega}, \mathbb{R}^2)$ be such that for a.e. $x$ in $\Omega$,
$$a_1a_2<\det D \varphi(x)<b_1b_2,$$
$$\lambda_2(D \varphi(x))< \lambda_1(D \varphi(x))\frac{b_2-a_2}{b_1-a_1}+\frac{a_2b_1-b_2a_1}{b_1-a_1}.$$
Then there exists
a map $u\in \varphi+ W^{1,\infty}_0(\Omega, \mathbb{R}^2)$ such that
$Du(x) \in E$  for a.e. $x$ in $\Omega$.
\end{theorem}

We refer to Theorems \ref{existence isotropic with determinant constraint n=2} and
\ref{existence isotropic with determinant constraint n=3} for more details. We stress the fact
that these results are independent of the knowledge of $\operatorname*{Rco}_fE$, which is not
known for $n>2$, and that analogous results could be
obtained in higher dimensions. In practical applications the conditions stated are easier to verify
than the conditions needed to characterize $\operatorname*{Rco}_fE$.

\section{Review on the generalized notions of convexity}\label{sectionreview}

In this section we recall several definitions and properties of some generalized notions of convexity that
will be useful throughout this paper. We refer to Dacorogna's monograph \cite{DirectMethods} and to
Dacorogna and Ribeiro \cite{Dacorogna-Ribeiro} for more details.

Several types of hulls in a generalized sense will be recalled here. The main result of this section is the
characterization of the hull $\operatorname*{Rco}_f E$, established in Theorem \ref{Rcofcharact}.

We start by recalling the notions of polyconvex and rank one convex functions.

\begin{notation} For $\xi\in\mathbb{R}^{N\times n}$ we let
\[
T\left(  \xi\right)  =\left(
\xi,\mathrm{adj}_{2}\xi,\ldots,\mathrm{adj}_{N\wedge n}\xi\right)
\in\mathbb{R}^{\tau(N,n)},
\]
where $\mathrm{adj}_{s}\xi$ stands for the matrix of all $s\times
s$ subdeterminants of the matrix $\xi,$ $1\leq s\leq N\wedge
n=\min\left\{  N,n\right\}  $ and where
\[
\tau=\tau\left(  N,n\right)  =\underset{s=1}{\overset{N\wedge
n}{\sum}}\binom {N}{s}\binom{n}{s}
 \text{ and }\binom{N}{s}
 =\frac{N!}{s!\left(  N-s\right)  !}.
\]
In particular, if $N=n=2,$ then $T\left(  \xi\right)  =\left(
\xi,\det \xi\right).$
\end{notation}

\begin{definition}\label{convfunctions}
(i) A function $f:\mathbb{R}^{N\times
n}\rightarrow\mathbb{R}\cup\left\{  +\infty\right\}  $ is said to
be \emph{polyconvex} if there exists a convex function
$g:\mathbb{R}^{\tau(N,n)}\to\mathbb{R}\cup\left\{
+\infty\right\}$ such that
$f(\xi)=g(T(\xi)).$

(ii) A function $f:\mathbb{R}^{N\times
n}\rightarrow\mathbb{R}\cup\left\{  +\infty\right\}  $ is said to
be \emph{rank one convex} if
\[
f\left(  \lambda\xi+(1-\lambda)\eta\right)  \leq \lambda\,f\left(
\xi\right) +\left(  1-\lambda\right) \,f\left( \eta\right)
\]
for every $\lambda\in\left[  0,1\right] $ and every
$\xi,\eta\in\mathbb{R}^{N\times n}$ with
$\operatorname*{rank}(\xi-\eta)  =1$.\smallskip
\end{definition}

It is well known that
$f \, \textnormal{polyconvex} \Rightarrow
f \, \textnormal{rank one convex}.$

Next we recall the corresponding notions of convexity for sets.

\begin{definition}\label{definitiongeneralizedconvexities}
(i) We say that $E\subset\mathbb{R}^{N\times n}$ is
\emph{polyconvex} if there exists a convex set $K\subset \mathbb{R}^{\tau(N,n)}$ such that
$\left\{\xi\in\mathbb{R}^{N\times n}: T(\xi)\in K\right\}=E.$

(ii) Let $E\subset\mathbb{R}^{N\times n}$. We say that $E$ is
\emph{rank one convex} if for every $\lambda\in[0,1]$ and
for every $\xi,\eta\in E$ such that $\operatorname*{rank}(\xi-\eta)=1$, then
$\lambda \xi+(1-\lambda)\eta\in E.$
\end{definition}

As shown by Dacorogna and Ribeiro \cite{Dacorogna-Ribeiro} a set $E$ is polyconvex if and only if the
following condition is satisfied, for every $I \in \mathbb{N}$
$$\left.\begin{array}{c}\vspace{0.2cm}\displaystyle{\sum_{i=1}^{I}\lambda_i
T(\xi_i)=T\left(\sum_{i=1}^{I}\lambda_i \xi_i\right)}\\
\xi_i\in E,\ \lambda_i \geq 0, \ \displaystyle{\sum_{i=1}^{I}\lambda_i = 1}
\end{array}\right\}\Rightarrow \sum_{i=1}^{I}\lambda_i \xi_i\in
E.$$

Moreover, we have the following
implication
\begin{center}$E$ polyconvex$\ \Rightarrow$ $E$ rank one convex.\end{center}

As in the classical convex case, for these convexity notions, related convex hulls can be considered.

\begin{definition}\label{convexhulls}
The \emph{polyconvex} and \emph{rank one convex hulls} of a set
$E\subset\mathbb{R}^{N\times n}$ are, respectively, the smallest
polyconvex and rank one convex sets containing $E$ and are, respectively, denoted by
$\operatorname*{Pco}E$ and $\operatorname*{Rco}E$.
\end{definition}

Obviously one has the following inclusions
$$E\subseteq \operatorname*{Rco}E \subseteq \operatorname*{Pco}E\subseteq\operatorname*{co}E,$$ where
$\operatorname*{co}E$ denotes the convex hull of $E$.

We recall the usual characterizations for the polyconvex and rank one convex hulls.
It was proved by Dacorogna and Marcellini in \cite{Dac-Marc} that
\begin{equation}
\label{pco}\operatorname*{Pco}E =\left\{\xi\in\mathbb{R}^{N\times n}:\,
\displaystyle T(\xi)=\sum_{i=1}^{\tau+1}t_i T(\xi_i),
\ \xi_i\in E, \ t_i\ge 0, \ \sum_{i=1}^{\tau+1}t_i=1\right\}
\end{equation}
and \begin{equation}\label{Rico}\displaystyle{\operatorname*{Rco}E=
\bigcup_{i\in\mathbb{N}}\mathrm{R}_i\mathrm{co}E,}\end{equation} where
$\mathrm{R}_0\mathrm{co}E=E$ and
$$\mathrm{R}_{i+1}\mathrm{co}E=\left\{\xi\in\mathbb{R}^{N\times n}:\
\begin{array}
[c]{c}\vspace{0.2cm} \xi=\lambda A+(1-\lambda) B,\ \lambda\in[0,1],\\
A,B\in \mathrm{R}_i\mathrm{co}E,\ \operatorname*{rank}(A-B)\leq 1
\end{array}\right\},\ i\ge 0.$$

One has (see \cite{Dacorogna-Ribeiro}) that $\operatorname*{Pco}E$ and
$\operatorname*{Rco}E$ are open if $E$ is open, and $\operatorname*{Pco}E$ is
compact if $E$ is compact. However, in general, it isn't true that $\operatorname*{Rco}E$ is
compact if $E$ is compact (see Kol\'a\v{r} \cite{Kolar}).

It is well known that, for $E\subset\mathbb{R}^{N\times n}$,
\begin{equation}\label{coE}
\operatorname*{co}E =\left\{\xi\in\mathbb{R}^{N\times n}: f(\xi) \leq0,\ \text{for every convex function }
f\in{\mathcal{F}}^{E}_\infty\right\}\end{equation}
\begin{equation}\label{closurecoE}\overline{\operatorname*{co}E}  =\left\{ \xi\in\mathbb{R}^{N\times n}:
f(\xi)  \leq0,\ \text{for every convex function }f\in\mathcal{F}^{E}\right\}
\end{equation}
where $\overline{\operatorname*{co}E}$ denotes the closure of the
convex hull of $E$ and
\begin{align*}
{\mathcal{F}}^{E}_\infty  &  =\left\{  f:\mathbb{R}^{N\times n}
\rightarrow\mathbb{R}\cup\left\{ +\infty\right\}
:\left.  f\right\vert _{E}\leq0\right\} \\
\mathcal{F}^{E}  &  =\left\{  f:\mathbb{R}^{N\times n}\rightarrow
\mathbb{R}:\left.  f\right\vert _{E}\leq0\right\}  .
\end{align*}

Analogous representations to $(\ref{coE})$ can be obtained in the
polyconvex and rank one convex cases:
\begin{align*}
&\operatorname*{Pco}E  = \left\{ \xi\in\mathbb{R}^{N\times n}:
f(\xi) \leq0,\ \text{for every polyconvex function }
f\in{\mathcal{F}}^{E}_\infty\right\}, \\
&\operatorname*{Rco}E =\left\{ \xi\in\mathbb{R}^{N\times n}:
f(\xi) \leq0,\ \text{for every rank one convex function }
f\in{\mathcal{F}}^{E}_\infty\right\}.
\end{align*}
However,
$(\ref{closurecoE})$ can only be generalized to the polyconvex
case if the sets are compact, and, in the rank one convex case,
$(\ref{closurecoE})$ is not true, even if compact sets are
considered.
In view of this, another type of hulls can be defined.

\begin{definition}\label{hullsfinite}
For a set $E$ of $\mathbb{R}^{N\times n}$, let
\begin{align*}
\operatorname*{co}\nolimits_{f}E  & =\left\{
\xi\in\mathbb{R}^{N\times n}:f\left(
\xi\right)  \leq0\text{, for every convex }f\in\mathcal{F}^{E}\right\}\\
\operatorname*{Pco}\nolimits_{f}E  &  =\left\{
\xi\in\mathbb{R}^{N\times n}:f\left(
\xi\right)  \leq0\text{, for every polyconvex }f\in\mathcal{F}^{E}\right\} \\
\operatorname*{Rco}\nolimits_{f}E  &  =\left\{
\xi\in\mathbb{R}^{N\times n}:f\left( \xi\right)  \leq0\text{, for
every rank one convex
}f\in\mathcal{F}^{E}\right\}.
\end{align*}
\end{definition}

\begin{remark}
1) Notice that these hulls are closed sets.
Moreover, they are, respectively, convex, polyconvex and rank one convex.

2) For compact sets $E$, these are the hulls considered by M\"uller and {\v{S}}ver{\'a}k
\cite{Muller-Sverakcounterexamples} to establish an existence result for differential inclusions. We notice
that a different
definition was introduced for open sets.
\end{remark}

Thus, as observed above, $\overline{\operatorname*{co}E}=\operatorname*{co}\nolimits_{f}E;$
if $E$ is compact, then
$$\operatorname*{Pco}E=\overline{\operatorname*{Pco}E}=\operatorname*{Pco}\nolimits_{f}E,$$ but, in
general,
$$
\operatorname*{Pco}E\subsetneq\overline{\operatorname*{Pco}E}
\subsetneq \operatorname*{Pco}\nolimits_{f}E.
$$
Moreover, in general, even if $E$ is compact,
$$\operatorname*{Rco}E\subsetneq\overline{\operatorname*{Rco}E}
\subsetneq \operatorname*{Rco}\nolimits_{f}E.$$

Next we establish a characterization of the hull $\operatorname*{Rco}\nolimits_{f} E$ for a given compact set
$E$. Based on the following result, we will investigate in Section \ref{sectionsufficientconditions}
sufficient conditions for the relaxation property (cf. Definition \ref{prop-relax}) which is the key to apply
the Baire categories method for vectorial differential inclusions due to Dacorogna and Marcellini \cite{Dac-Marc}.

Before stating the result we give a definition.

\begin{definition} \label{defH_I}
Let $U$ be a subset of $\mathbb{R}^{N\times n}$ and, for some integer $I\ge 1$, let $\xi_i\in
\mathbb{R}^{N\times n}$ and $\lambda_i>0,\ i=1,...,I$ be such that $\sum_{i=1}^I\lambda_i=1$. We say that
$(\lambda_i,\xi_i)_{1\le i\le I}$ satisfy $(H_I(U))$ if  \smallskip

(i) in the case $I=1$, $\xi_1\in U$;\smallskip

(ii) in the case $I=2$, $\xi_1,\xi_2\in U$ and $\operatorname*{rank}(\xi_1-\xi_2)=1$;\smallskip

(iii) in the case $I>2$, up to a permutation, $\xi_1,\xi_2\in U$, $\operatorname*{rank}(\xi_1-\xi_2)=1$ and
defining
$$\left\{\begin{array}{ll}\mu_1=\lambda_1+\lambda_2, &
\eta_1=\frac{\lambda_1\xi_1+\lambda_2\xi_2}{\lambda_1+\lambda_2}\vspace{0.2cm}\\
\mu_i=\lambda_{i+1}, & \eta_i=\xi_{i+1},\ 2\le i\le I-1
\end{array}\right.$$
then $(\mu_i,\eta_i)_{1\le i\le I-1}$ satisfy $(H_{I-1}(U))$.
\end{definition}

\begin{remark}
The property in the above definition was introduced in \cite[page 174]{DirectMethods}, but here we have
the additional condition that the vertices of the ``chain'' must be elements of a given set. Moreover, we
notice that in the above definition, in particular, all $\xi_i\in U$.
\end{remark}

\begin{theorem}\label{Rcofcharact}
Let $E\subset\mathbb{R}^{N\times n}$ be a compact set and let $U$ be an open and bounded subset of
$\mathbb{R}^{N\times n}$
containing $\operatorname*{Rco}_f E$. Then
\begin{equation}\label{rcof_with_property}
\mathrm{Rco}_f E=\left\{\begin{array}{l}\xi\in\mathbb{R}^{N\times n}:\ \forall\ \varepsilon>0\ \exists\
I\in\mathbb{N},\ \exists\
 (\lambda_i,\xi_i)_{i=1,...,I} \text{ with } \lambda_i>0,\vspace{0.2cm}\\  \displaystyle{
 \sum_{i=1}^I\lambda_i=1} \text{ satisfying } (H_I(U)),\ \xi=\sum_{i=1}^I\lambda_i\xi_i,\
\sum_{\substack{i=1\\\xi_i\notin
B_\varepsilon(E)}}^I\lambda_i<\varepsilon\end{array} \right\},
\end{equation}
where $B_\varepsilon(E)=\{\xi\in \mathbb{R}^{N\times n}:\ \mathrm{dist}(\xi;E)<\varepsilon\}$.
\end{theorem}

\begin{remark}\label{remark Rco characterization}
1) The fact that $\operatorname*{Rco}_f E$ is included in the set on the right hand side of (\ref{rcof_with_property}) was obtained
by means of Young measures by M\"uller and {\v{S}}ver{\'a}k
\cite[Theorem 2.1]{Muller-Sverakcounterexamples} as a
refinement of a result due to Pedregal \cite{Pedregal}. Below we recall the proof without mentioning Young
measures.

2) Since, for compact sets $E$, $\operatorname*{Rco}_f E\subseteq\operatorname*{co} E$, the set $U$ can be
chosen to be any convex open set containing $E$.

3) This result should be compared with the following characterization of $\operatorname*{Rco}E$ which follows
trivially from (\ref{Rico}): for any set $E\subset\mathbb{R}^{N\times n}$,
$$\mathrm{Rco}E=\left\{\begin{array}{l}\xi\in\mathbb{R}^{N\times n}:\ \exists\ I\in\mathbb{N},\ \exists\
 (\lambda_i,\xi_i)_{i=1,...,I} \text{ with } \lambda_i>0,\ \displaystyle{
 \sum_{i=1}^I\lambda_i=1}\vspace{0.2cm}\\  \text{ satisfying } (H_I(\operatorname*{Rco}E)),\
\displaystyle{\xi=\sum_{i=1}^I\lambda_i\xi_i,\ \
\xi_i\in E,\ \forall\ i=1,...,I}\end{array} \right\}.$$
\end{remark}

\begin{proof}
Let us call $X$ the set on the right hand side of the identity to be proved.

First we show that $X\subseteq \operatorname*{Rco}_f E$. Let $\xi\in X$ and let
$f:\mathbb{R}^{N\times n}\to\mathbb{R}$
be any rank one convex function such that $f_{|E}\le 0$. We will show that $f(\xi)\le 0$
by verifying that $f(\xi)\le \delta$ for any $\delta>0$.

We start by noticing that, since $f$ is continuous, it is uniformly continuous in $\overline{U}$.
Thus, fix $\delta>0$ and let $\gamma>0$ be such that
$$\forall\ \eta_1,\eta_2\in \overline{U},\ |\eta_1-\eta_2|\le \gamma\ \Rightarrow\
|f(\eta_1)-f(\eta_2)|\le \delta.$$

Now, let $\varepsilon>0$ be such that $\varepsilon\le \min\{\delta,\gamma\}$ and
$B_\varepsilon(E)\subset \overline{U}$. By definition of $X$, let $I_\varepsilon\in\mathbb{N},\
(\lambda_i^\varepsilon,\xi_i^\varepsilon)_{i=1,...,I_\varepsilon}$ with $\lambda_i^\varepsilon>0$,
$\displaystyle\sum_{i=1}^{I_\varepsilon}\lambda_i^\varepsilon=1$ satisfying $(H_{I_\varepsilon}(U))$,  be
such that $\displaystyle{\xi=\sum_{i=1}^{I_\varepsilon}\lambda_i^\varepsilon\xi_i^\varepsilon,\
\sum_{\substack{i=1\\ \xi_i^\varepsilon\notin
B_\varepsilon(E)}}^{I_\varepsilon}\lambda_i^\varepsilon<\varepsilon}$. Then, using the rank one convexity of $f$, we have
$$f(\xi)=f\left(\sum_{i=1}^{I_\varepsilon}\lambda_i^\varepsilon\xi_i^\varepsilon\right)\le
\sum_{i=1}^{I_\varepsilon}\lambda_i^\varepsilon f(\xi_i^\varepsilon)=\sum_{\substack{i=1\\
\xi_i^\varepsilon\notin
B_\varepsilon(E)}}^{I_\varepsilon}\lambda_i^\varepsilon f(\xi_i^\varepsilon)+\sum_{\substack{i=1\\
\xi_i^\varepsilon\in
B_\varepsilon(E)}}^{I_\varepsilon}\lambda_i^\varepsilon f(\xi_i^\varepsilon).$$
We are going to estimate the last two sums. For the first one we have
$$\sum_{\substack{i=1\\ \xi_i^\varepsilon\notin
B_\varepsilon(E)}}^{I_\varepsilon}\lambda_i^\varepsilon f(\xi_i^\varepsilon)\le C \sum_{\substack{i=1\\
\xi_i^\varepsilon\notin
B_\varepsilon(E)}}^{I_\varepsilon}\lambda_i^\varepsilon\le C\delta,$$
where $C := \underset{\overline{U}}\max f$.
For the second sum we use the uniform continuity of $f$ in $\overline{U}$. Since
$\xi_i^\varepsilon\in B_\varepsilon(E)$, we can consider $\eta_i^\varepsilon\in E$ such that
$|\eta_i^\varepsilon-\xi_i^\varepsilon|<\varepsilon$. Then,
$$f(\xi_i^\varepsilon)\le f(\xi_i^\varepsilon)-f(\eta_i^\varepsilon)\le
|f(\xi_i^\varepsilon)-f(\eta_i^\varepsilon)|\le \delta.$$

We then conclude that $f(\xi)\le (1+C)\delta$. Since $\delta$ is arbitrarily small, we obtain $f(\xi)\le 0$, as we wanted.

We will now prove the other inclusion, $\operatorname*{Rco}_f E\subseteq X$.
We suppose by contradiction
that $\xi\in\operatorname*{Rco}_f E$ and $\xi\notin X$. Then, there exists $\varepsilon>0$ such that, for
every $I\in\mathbb{N}$ and for every $(\lambda_i,\xi_i)$ satisfying $(H_I(U))$ with
$\xi=\sum_{i=1}^I \lambda_i\xi_i$ we have
$$\sum_{\substack{i=1\\ \xi_i\notin B_\varepsilon(E)}}^{I}\lambda_i\ge\varepsilon.$$

Defining, for $\eta \in U$, $f(\eta):=\mathrm{dist}(\eta;E)$ and
$$g(\eta):=\inf\left\{\begin{array}{rl}\displaystyle\sum_{i=1}^I\mu_if(\eta_i):& I\in\mathbb{N},\
(\mu_i,\eta_i)_{i=1,...,I} \text{ with } \mu_i>0,\ \displaystyle{\sum_{i=1}^I\mu_i=1}  \vspace{0.2cm}\\ &
\text{satisfying }(H_I(U))\end{array} \right\},$$
it is trivial to see that $0\le g<+\infty$ and $g_{|E}= 0$. Moreover, for any $(\mu_i,\eta_i)_{i=1,...,I}$
as in the definition of $g(\xi)$, the contradiction assumption gives
$$\sum_{i=1}^I\mu_if(\eta_i)\ge \sum_{\substack{i=1\\ \eta_i\notin
B_\varepsilon(E)}}^I\mu_if(\eta_i)\ge \sum_{\substack{i=1\\ \eta_i\notin
B_\varepsilon(E)}}^I\mu_i\varepsilon\ge\varepsilon^2>0$$
and therefore $g(\xi)>0$. Finally, we show that $g$ is rank one convex on $U$ according to the
definition in \cite{Muller-Sverakcounterexamples}, that is, if $A,B\in U$ with
$\mathrm{rank}(A-B)=1$ and $\lambda A+(1-\lambda)B\in U,\ \forall\ \lambda\in(0,1)$ then
$g(\lambda A+(1-\lambda)B)\le \lambda g(A)+(1-\lambda)g(B)$.
To achieve this, it is enough to observe that if
$(\mu_i,A_i)_{i=1,...,I}$ and $(\gamma_i,B_i)_{i=1,...,J}$ are as in the definition of $g(A)$ and $g(B)$,
respectively, then $((\lambda\mu_i,A_i), ((1-\lambda)\gamma_i,B_i))$ satisfy the conditions in the
definition of $g(\lambda A+(1-\lambda)B)$.

These properties of $g$ allow us to apply the extension result \cite[Lemma 2.3]{Muller-Sverakcounterexamples}
which ensures that there exists a rank one convex function
$G:\mathbb{R}^{N\times n}\to \mathbb{R}$
coinciding with $g$ in a neighborhood of $\operatorname*{Rco}_f E$. This yields the desired contradiction,
since we are assuming that $\xi\in \operatorname*{Rco}_f E$.
\end{proof}

\section{Sufficient conditions for the relaxation property}\label{sectionsufficientconditions}

The Baire categories method, developed by Dacorogna and Marcellini \cite{Dac-Marc} for solving vectorial
differential
inclusions, relies on a fundamental property, called relaxation property, cf. Definition \ref{prop-relax}
below. Due to the difficulty in dealing with this property in the applications, sufficient conditions
for it were also obtained in \cite{Dac-Marc}. They ensure existence of solutions to
the differential inclusion boundary value problem when the gradient of the boundary data is in
the interior of the rank one convex hull of the set where the differential inclusion is to be solved. However,
in some examples this hull turns out to be too restrictive. Therefore, our goal in this section is to find
more flexible sufficient conditions for the relaxation property. This will allow us to handle the
problems considered in Section \ref{sectionexamples}.

We start by recalling the relaxation property introduced by Dacorogna and Marcellini \cite{Dac-Marc} and their
related existence theorem for differential inclusions, here
in a more general version due to Dacorogna and Pisante \cite{Dacorogna-Pisante}.

\begin{definition}[Relaxation Property]\label{prop-relax} Let $E$,
$K\subset\mathbb{R}^{N\times n }$. We say that $K$ has the relaxation property with respect to $E$ if, for
every bounded open set  $\Omega\subset\mathbb{R}^n$ and for every affine function $u_\xi$, such that
$Du_\xi(x)=\xi$ and
$Du_\xi(x)\in K,$ there exists a sequence $u_\nu\in Aff_{piec}(\overline{\Omega};\mathbb{R}^N)$ such that
$$\begin{array}{l} u_\nu\in
u_\xi+W_0^{1,\infty}(\Omega;\mathbb{R}^N),\quad
Du_\nu(x)\in E\cup K,\ a.e.\ x\text{ in } \Omega,\vspace{0.2cm}\\
u_\nu\stackrel{*}{\rightharpoonup}u_\xi\ \mbox{in}\
W^{1,\infty}(\Omega;\mathbb{R}^N),\quad
\displaystyle{\lim_{\nu\rightarrow+\infty}\int_\Omega
\mathrm{dist}(Du_\nu(x);E)\,dx =0.}
\end{array}$$
\end{definition}

\begin{theorem}\label{abstract existence theorem}
Let
$\Omega\subset\mathbb{R}^{n}$ be open and bounded. Let $E\subset
\mathbb{R}^{N\times n}$ and $K\subset
\mathbb{R}^{N\times n}$ be   compact and  bounded, respectively.
Assume that $K$ has the relaxation property with respect to $E$. Let
$\varphi\in Aff_{piec}\left(  \overline{\Omega};\mathbb{R}^{N}\right)  $ be
such that
\[
D\varphi\left(  x\right)  \in E\cup K\text{, a.e. } x \text{ in }\Omega.
\]
Then there exists (a dense set of) $u\in\varphi+W_{0}^{1,\infty}\left(
\Omega;\mathbb{R}^{N}\right)  $ such that
\[
Du\left(  x\right)  \in E\text{, a.e. } x \text{ in }\Omega.
\]
Moreover, if $K$ is open, $\varphi$ can be taken in $C^1_{piec}(\overline{\Omega}; \mathbb{R}^N).$
\end{theorem}

A sufficient condition for the relaxation property is the approximation property \cite[Definition 6.12 and
Theorem 6.14]{Dac-Marc} (see also \cite[Definition 10.6 and Theorem 10.9]{DirectMethods}) that we recall next.

\begin{definition}
[Approximation property]\label{Approximation property}Let $E\subset K  \subset\mathbb{R}^{N\times n}.$ The
sets $E$ and $K$ are said to have the\textit{\ }approximation property if there exists a
family of closed sets $E_{\delta}$ and $K_\delta$,
$\delta>0$, such that\smallskip

(i) $E_{\delta}\subset K_\delta \subset\operatorname*{int}%
K$ for every $\delta>0;$\smallskip

(ii)  $\forall\ \varepsilon >0\ \exists\ \delta_0>0:\
\mathrm{dist}(\eta; E)\le\varepsilon,\ \forall\ \eta\in E_\delta,\
\delta\in(0,\delta_0]$;\smallskip

(iii) $\eta\in \mathrm{int}\,K\ \Rightarrow\ \exists\ \delta_0>0:\
\eta\in K_\delta,\ \forall\ \delta\in (0,\delta_0]$.
\end{definition}

We therefore have the following theorem.

\begin{theorem}
\label{Relaxation property theorem for RcoE}Let $E\subset\mathbb{R}^{N\times
n}$ be compact and assume $\operatorname*{Rco}E$ has the approximation property with
$K_\delta =\operatorname*{Rco}E_{\delta}$. Then
$\operatorname*{int}\operatorname*{Rco}E$ has the relaxation property with
respect to $E$.
\end{theorem}

In the spirit of the approximation property, we establish a sufficient condition for the relaxation property
such that larger sets than the rank one convex hull of $E$ can be considered as the set $K$ in
Theorem \ref{abstract existence theorem}. More precisely, we will show that hulls like
$\operatorname*{Rco}_f E$ are likely to enjoy the relaxation property.

We can now state our main theorem.

\begin{theorem}\label{sufficientconditionforrelaxation}
Let $E, K$ be two bounded subsets of $\mathbb{R}^{N\times n}$ and let, for $\delta>0$, $E_\delta,K_\delta$ be
sets verifying the following conditions:\smallskip

(i) $\forall\ \varepsilon >0\ \exists\ \delta_0>0:\
\mathrm{dist}(\eta; E)\le\varepsilon,\ \forall\ \eta\in E_\delta,\
\delta\in(0,\delta_0]$;\smallskip

(ii) $\eta\in \mathrm{int}\,K\ \Rightarrow\ \exists\ \delta_0>0:\
\eta\in K_\delta,\ \forall\ \delta\in (0,\delta_0]$;

(iii) $\forall\ \delta>0\ \forall\ \xi\in K_\delta\ \exists\ I\in\mathbb{N},\ \exists\
(\lambda_i,\xi_i)_{1\le i\le I} \text{ with } \lambda_i>0,\     \displaystyle\sum_{i=1}^I\lambda_i=1,$

$\hspace{0.7cm}\xi_i\in\mathbb{R}^{N\times n}, \text{ satisfying } (H_I(\mathrm{int}\,K))\text{ and } \displaystyle{\xi=\sum_{i=1}^I\lambda_i\xi_i,\
\sum_{\substack{i=1\\ \xi_i\notin
E_\delta}}^I\lambda_i<\delta}.$

Then $\operatorname*{int}K$ has the relaxation property with respect to $E$.
\end{theorem}

\begin{remark} 1) In the sense of Proposition \ref{approximpliesC},
Theorem \ref{sufficientconditionforrelaxation}
is a generalization of the usual approximation property for
the hulls $\operatorname*{Rco}_f$.

2) This result is analogous to Theorem 6.15 in \cite{Dac-Marc} and allows us to work with subsets of
$\operatorname*{Rco}\nolimits_{f}E$.
\end{remark}

Before proving the result we establish two corollaries.

\begin{corollary}\label{corol1}
Let $E\subset\mathbb{R}^{N\times n}$ be bounded. Assume that there exist compact sets
$E_\delta\subset\mathbb{R}^{N\times n}$
such that, defining $K_\delta=\operatorname*{Rco}\nolimits_{f}E_\delta$ and
$K=\operatorname*{Rco}\nolimits_{f}E$, $K_\delta\subset\operatorname*{int}K$ and conditions $(i)$ and
$(ii)$ of Theorem \ref{sufficientconditionforrelaxation} are satisfied. Then
$\mathrm{int}\operatorname*{Rco}\nolimits_{f}E$ has the relaxation property with respect to $E$.
\end{corollary}

\begin{remark} One can easily see that Theorem \ref{sufficientconditionforrelaxation} is also true if in
condition $(iii)$ we replace $E_\delta$ by $B_\delta(E)$. In this case, Corollary \ref{corol1} follows
directly from Theorem \ref{sufficientconditionforrelaxation} and Proposition \ref{approximpliesC} below. In
the proof that we present here we have to consider artificial approximating sets $E_\delta$.
\end{remark}

\noindent\textbf{Proof of Corollary \ref{corol1}.}
Let $\widetilde{E}_\delta=B_\delta(E_\delta)$ and let
$\widetilde{K}_\delta=\operatorname*{Rco}\nolimits_{f}E_\delta$. We will show the result as an application of
Theorem \ref{sufficientconditionforrelaxation} for this choice of approximating sets.

By the hypotheses on $E_\delta$ and $K_\delta$ one can easily see that conditions $(i)$ and $(ii)$ of Theorem
\ref{sufficientconditionforrelaxation} are still satisfied by $\widetilde{E}_\delta$ and
$\widetilde{K}_\delta$. Condition $(iii)$ follows from Theorem \ref{Rcofcharact} applied to
$\operatorname*{Rco}\nolimits_{f}E_\delta$, noticing that we are assuming that $E_\delta$ is compact,
and taking
$U=\operatorname*{int}\operatorname*{Rco}\nolimits_{f}E$ which contains
$\operatorname*{Rco}\nolimits_{f}E_\delta$ by hypothesis.
$\hfill\Box$\medskip

The following result was already proved by Ribeiro \cite{Ribeiro-thesis}.

\begin{corollary}\label{corol2}
Let $E, K\subset\mathbb{R}^{N\times n}$ be such that $E$ is compact and $K$ is bounded. Assume that the following
condition holds:

$(H)$ given $\delta>0$, there exists
$L=L(\delta,E,K)\in\mathbb{N}$ such that
$$\begin{array}{l}\forall\ \xi\in \operatorname*{int}K\setminus B_{\delta}(E) \vspace{0.2cm}\\ \exists \
\eta_1,...,\eta_J\in \mathbb{R}^{N\times n},\ J\in\mathbb{N},\
J\le L,\  \operatorname*{rank}(\eta_j)=1,\ j=1,...,J\vspace{0.2cm}\\
\left[\xi+\eta_1+...+\eta_{j-1}-\eta_j,\xi+\eta_1+...+\eta_{j-1}+\eta_j\right] \subset \operatorname*{int}K,\
j=1,...,J,\vspace{0.2cm}\\
\xi+\eta_1+...+\eta_J\in B_{\delta}(E),
\end{array}
$$
where $[A,B]$ represents the segment joining the matrices $A$ and $B$.
Then $\operatorname*{int}K$ has the relaxation property with respect to $E$.
\end{corollary}

\begin{remark}\label{KsubsetRcofE}
Using the same ideas of the following proof, it turns out that, under the conditions of Corollary \ref{corol2}, the set $K$
is contained in $\operatorname*{Rco}_f E$.
\end{remark}

\noindent\textbf{Proof of Corollary \ref{corol2}.}
We will prove that
$$E_\delta=\mathrm{int}\,K\cap B_\delta(E)\qquad\text{and}\qquad
K_\delta=(\mathrm{int}\,K\cap E) \cup (\mathrm{int}\,K\setminus B_\delta(E))$$
satisfy conditions $(i)$, $(ii)$ and $(iii)$ of Theorem \ref{sufficientconditionforrelaxation}.

Condition $(i)$ is trivial. To get condition $(ii)$, we observe that if $\eta\in\mathrm{int}\,K$ then either
$\eta\in E$, and thus $\eta\in K_\delta$ for every $\delta>0$, or $\eta\notin E$. In this last case, since $E$
 is compact, $\mathrm{dist}\,(\eta;E)>0$ which entails $(ii)$.

It remains to show condition $(iii)$. Let $\delta>0$, $\xi\in K_\delta$ and consider $L=L(\delta, E,
K)\in\mathbb{N}$ as in the hypothesis. If $\xi\in \mathrm{int}\,K\cap E$, then condition $(iii)$
is satisfied with $I=1$ and
$(\lambda_i,\xi_i)_{i=1}=(1,\xi)$ and we are left with the case $\xi\in\mathrm{int}\,K\setminus B_\delta(E)$.
Applying the hypothesis, we have
$$\begin{array}{rl}&\exists \ \eta_1,...,\eta_J\in \mathbb{R}^{N\times n},\ J\in\mathbb{N},\
J\le L,\  \operatorname*{rank}(\eta_j)=1,\ j=1,...,J\vspace{0.2cm}\\ &
[\xi+\eta_1+...+\eta_{j-1}-\eta_j,\xi+\eta_1+...+\eta_{j-1}+\eta_j]\subset \operatorname*{int}K,\
j=1,...,J,\vspace{0.2cm}\\&
 \xi+\eta_1+...+\eta_J\in B_{\delta}(E).\end{array}
$$

Thus, by iteratively writing convex combinations using the matrices $\eta_i$, $i = 1, \cdots,J$, we obtain
\begin{equation}\label{eqcorol}\xi=\frac{1}{2}(\xi-\eta_1)+\frac{1}{2^2}(\xi+\eta_1-\eta_2)+...
+\frac{1}{2^J}(\xi+\eta_1+...+\eta_{J-1}-\eta_J)+
\frac{1}{2^J}(\xi+\eta_1+...+\eta_J).
\end{equation}
We notice that if we take $$\left\{\begin{array}{l}\lambda_j=\frac{1}{2^j},\ \text{if}\
1\le j\le J\vspace{0.2cm}\\
\lambda_{J+1}=\frac{1}{2^J}\end{array}\right.\; \; \text{and} \; \; \left\{\begin{array}{l}
\xi_j=\xi+\eta_1+...+\eta_{j-1}-\eta_j,\ \text{if}\ 1\le j\le J\vspace{0.2cm}\\
\xi_{J+1}=\xi+\eta_1+...+\eta_J,\end{array}\right.$$
then $(\lambda_j,\xi_j)_{1\le j\le J+1}$ satisfy $(H_{J+1}(\mathrm{int}\,K))$ and (\ref{eqcorol}) can be
rewritten in the form
$$\xi=\sum_{j=1}^{J+1}\lambda_j\xi_j,\quad\text{with}\quad\sum_{\substack{j=1\\ \xi_j\notin
B_\delta(E)}}^{J+1}\lambda_j\le 1-\frac{1}{2^J}\le 1-\frac{1}{2^L}.$$
If all $\xi_j\in B_\delta(E)$, then
$\displaystyle\sum_{\substack{j=1\\ \xi_j\notin B_\delta(E)}}^{J+1}\lambda_j=0$
and we have achieved condition $(iii)$. Otherwise, for each $\xi_j\in\mathrm{int}\,K\setminus B_\delta(E)$ we
apply again the hypothesis and get, for some $I_j\le L$,
$$\xi_j=\sum_{l=1}^{I_j+1}\widetilde{\lambda}_l^j\widetilde{\xi}_l^j\ \text{ such that }\
(\widetilde{\lambda}_l^j,\widetilde{\xi}_l^j)_{1\le l\le I_j+1}\text{ satisfy }(H_{I_j+1}(\mathrm{int}\,K))$$
and $$\sum_{\substack{l=1\\ \widetilde{\xi}_l^j\notin
B_\delta(E)}}^{I_j+1}\widetilde{\lambda}_l^j\le 1-\frac{1}{2^L}.$$
Therefore
$$
\xi=\sum_{\substack{j=1\\ \xi_j\in
B_\delta(E)}}^{J+1}\lambda_j\xi_j+\sum_{\substack{j=1\\ \xi_j\notin
B_\delta(E)}}^{J+1}\sum_{l=1}^{I_j+1}\lambda_j\widetilde{\lambda}_l^j\widetilde{\xi}_l^j\,,
$$
where the scalars and matrices in the above expression satisfy $(H_{\tilde{I}}(\mathrm{int}\,K))$
for a certain $\tilde{I} \in \mathbb{N}$, and
$$\sum_{\substack{j=1\\ \xi_j\notin
B_\delta(E)}}^{J+1}\sum_{\substack{l=1\\ \widetilde{\xi}_l^j\notin
B_\delta(E)}}^{I_j+1}\lambda_j\widetilde{\lambda}_l^j\le\sum_{\substack{j=1\\ \xi_j\notin
B_\delta(E)}}^{J+1}\lambda_j\left(1-\frac{1}{2^L}\right)\le\left(1-\frac{1}{2^L}\right)^2.$$
Of course, after a finite number of iterations of this procedure one gets condition $(iii)$.
$\hfill\Box$
\medskip

In order to prove Theorem \ref{sufficientconditionforrelaxation} we will show the following lemma.

\begin{lemma}\label{lemma}
Let $I>1$ be an integer and let $U\subset\mathbb{R}^{N\times n}$ be an open set. For $1\le i\le I$, let
$\xi_i\in \mathbb{R}^{N\times n}$ and $\lambda_i>0$ be such that $\sum_{i=1}^I\lambda_i=1$ and
$(\lambda_i,\xi_i)_{1\le i\le I}$ satisfy $(H_I(U))$. Denote by $\xi$ the sum
$\sum_{i=1}^I\lambda_i\xi_i$ and let
$u_{\xi}$ be an affine map such that $Du_{\xi} = \xi$. Then, for any
given $\varepsilon>0$ and any bounded open set $\Omega\subset\mathbb{R}^n$, there exist
$u_\varepsilon\in Aff_{piec}(\overline{\Omega};\mathbb{R}^N)$ and disjoint open sets
$\Omega_\varepsilon^i\subset\Omega$ such that
$$\begin{array}{l} u_\varepsilon\in
u_\xi+W_0^{1,\infty}(\Omega;\mathbb{R}^N),\vspace{0.2cm}\\
Du_\varepsilon(x)\in U\cup B_\varepsilon(\xi),\ a.e.\ x\in\Omega,\quad
Du_\varepsilon(x)=\xi_i,\ a.e.\ x\in\Omega_\varepsilon^i,\ i=1,...,I,\vspace{0.2cm}\\
||u_\varepsilon -u_\xi||_{L^{\infty}}\le\varepsilon,\quad
|\operatorname*{meas}(\Omega_\varepsilon^i)-\lambda_i\operatorname*{meas}(\Omega)|\le\varepsilon,\ i=1,...,I.
\end{array}$$
\end{lemma}

The proof of the lemma relies on the following approximation result, due to M\"{u}ller and Sychev
\cite[Lemma 3.1]{Muller-Sychev}, which is a refinement of a classical result.
For $a \in \mathbb{R}^N$ and $b \in \mathbb{R}^n$ we will denote by $a\otimes b$ the $N\times n$ matrix whose $(i,j)$ entry is $a_i b_j$.
\begin{lemma}
[Approximation lemma]\label{key-lemma}Let
$\Omega\subset\mathbb{R}^{n}$ be a bounded open set. Let
$A,B\in\mathbb{R}^{N\times n}$
be such that
$
A-B=a\otimes b
$,
with $a\in\mathbb{R}^{N}$\ and $b\in\mathbb{R}^{n}$. Let
$b_{3},\dots,b_{k} \in\mathbb{R}^{n},$ $k\geq n$, be such that
$0\in\operatorname*{int}\operatorname*{co}\{b,-b,b_{3},\dots,b_{k}\}$ and, for $t\in\lbrack0,1]$,
let $\varphi$ be an affine map such that
\[
D\varphi(x)=\xi=tA+(1-t)B,\ x\in\overline{\Omega}
\]
(i.e. $A=\xi+\left(  1-t\right)  a\otimes b$ and $B=\xi-ta\otimes
b$). Then, for every $\varepsilon>0,$ there exists
$u \in Aff_{piec}(\overline{\Omega};\mathbb{R}^N)$ and there exist disjoint open sets
$\Omega_{A},\Omega_{B}\subset\Omega,$ such that
\[
\left\{
\begin{array}
[c]{l}
\left\vert \operatorname*{meas}\,(\Omega_{A})-
t\;\operatorname*{meas}\,(\Omega)\right\vert \leq \varepsilon ,\;
\left\vert \operatorname*{meas}\,(\Omega_{B})
-(1-t)\,\operatorname*{meas}\,(\Omega)\right\vert \leq\varepsilon\\
u(x)=\varphi(x),\ x\in\partial\Omega\\
\left\vert u(x)-\varphi(x)\right\vert \leq\varepsilon,\ x\in\Omega\\
Du(x)=\left\{
\begin{array}
[c]{ll}
A & \text{in }\Omega_{A}\\
B & \text{in }\Omega_{B}
\end{array}
\right.  \\
Du(x)\in\xi+\{\left(  1-t\right)  a\otimes b,-ta\otimes b,a\otimes
b_{3},\dots,a\otimes b_{k}\},\ \text{a.e. }x\text{ in }\Omega.
\end{array}
\right.
\]
\end{lemma}

\medskip

\noindent\textbf{Proof of Lemma \ref{lemma}.} We prove the result by induction on $I$.

If $I=2$ it suffices to apply Lemma \ref{key-lemma} choosing $|b_3|,..., |b_k|$ sufficiently
small so that $|a \otimes b_i| \leq \varepsilon$ for $i = 3,..., k$.

Now, let $I>2$ and consider $(\lambda_i,\xi_i)_{1\le i\le I}$ as in the hypothesis.
Up to a permutation, and defining
$$\left\{\begin{array}{ll}\mu_1=\lambda_1+\lambda_2, &
\eta_1=\frac{\lambda_1\xi_1+\lambda_2\xi_2}{\lambda_1+\lambda_2}\vspace{0.2cm}\\
\mu_i=\lambda_{i+1},& \eta_i=\xi_{i+1},\ 2\le i\le I-1,
\end{array}\right.$$ we have $\xi_1,\xi_2\in U$,
$\operatorname*{rank}(\xi_1-\xi_2)=1$ and $(\mu_i,\eta_i)_{1\le i\le I-1}$ satisfy $(H_{I-1}(U))$.
Then the induction hypothesis ensures the existence of
$v\in Aff_{piec}(\overline{\Omega};\mathbb{R}^N)$ and disjoint open sets
$\widetilde{\Omega}_\varepsilon^i\subset\Omega,\ i=1,...,I-1$ such that
$$\begin{array}{l} v\in
u_\xi+W_0^{1,\infty}(\Omega;\mathbb{R}^N),\vspace{0.2cm}\\
Dv(x)\in U\cup B_\varepsilon(\xi),\ a.e.\ x\in\Omega,\quad
Dv(x)=\eta_i,\ a.e.\ x\in \widetilde{\Omega}_\varepsilon^i,\ i=1,...,I-1\vspace{0.2cm}\\
||v -u_\xi||_{L^{\infty}}\le\frac{\varepsilon}{2},\quad
|\text{meas}(\widetilde{\Omega}_\varepsilon^i)-\mu_i \text{meas}(\Omega)|\le\frac{\varepsilon}{2},\
i=1,...,I-1.
\end{array}$$

Since $(\mu_i,\eta_i)_{1\le i\le I-1}$ satisfy $(H_{I-1}(U))$, then $\eta_1\in U$. Now let
$0<\delta<\varepsilon$ be such that the neighborhood of $\eta_1$, $B_\delta(\eta_1)$, is contained in $U$
and apply again Lemma \ref{key-lemma} in $\widetilde{\Omega}_\varepsilon^1$ to obtain
$w\in Aff_{piec}(\overline{\widetilde{\Omega}_\varepsilon^1};\mathbb{R}^N)$ and disjoint open sets
$\Omega_\varepsilon^i\subset \widetilde{\Omega}_\varepsilon^1,\ i=1,2$ such that
\begin{eqnarray}\label{meas12}
& & \hspace{-0,2cm} w\in v+W_0^{1,\infty}(\widetilde{\Omega}_\varepsilon^1;\mathbb{R}^N),
\vspace{0.2cm} \nonumber\\
& & \hspace{-0,2cm} Dw(x)\in \{\xi_1,\xi_2\}\cup B_\delta(\eta_1),\
a.e.\ x\in \widetilde{\Omega}_\varepsilon^1,\quad
Dw(x)=\xi_i,\ a.e.\ x\in \Omega_\varepsilon^i,\ i=1,2\vspace{0.2cm} \nonumber\\
& & \hspace{-0,2cm} ||w -v||_{L^{\infty}}\le\frac{\varepsilon}{2},\quad
\left|\text{meas}(\Omega_\varepsilon^i)-\frac{\lambda_i}{\lambda_1+\lambda_2}
\text{meas}(\widetilde{\Omega}_\varepsilon^1)\right|\le\frac{\varepsilon}{2},\ i=1,2.
\end{eqnarray}

We then obtain the desired result taking $\Omega_\varepsilon^1$ and $\Omega_\varepsilon^2$ as above,
$\Omega_\varepsilon^i=\widetilde{\Omega}_\varepsilon^{i-1}$ for $i=3,...,I$ and
$$u_\varepsilon=\left\{\begin{array}{l}v\
\mathrm{in}\ \Omega\setminus\widetilde{\Omega}_\varepsilon^1,\vspace{0.2cm}\\w\
\mathrm{in}\ \widetilde{\Omega}_\varepsilon^1.
\end{array}\right.$$
In fact we only need to verify that
$|\text{meas}(\Omega_\varepsilon^i)-\lambda_i\text{meas}(\Omega)|\le\varepsilon$ for $i=1,2$:
$$\begin{array}{rcl}|\text{meas}(\Omega_\varepsilon^i)-\lambda_i\text{meas}(\Omega)|& \le &
\left|\text{meas}(\Omega_\varepsilon^i)-\frac{\lambda_i}{\lambda_1+\lambda_2}
\text{meas}(\widetilde{\Omega}_\varepsilon^1)\right|+\vspace{0.2cm} \\ & &
+\left|\frac{\lambda_i}{\lambda_1+\lambda_2}\text{meas}(\widetilde{\Omega}_\varepsilon^1)-
\lambda_i\text{meas}(\Omega)\right|\vspace{0.2cm}\\
& \le & \frac{\varepsilon}{2}+\frac{\lambda_i}{\lambda_1+\lambda_2}\left|
\text{meas}(\widetilde{\Omega}_\varepsilon^1)-(\lambda_1+\lambda_2)\text{meas}(\Omega)\right|\vspace{0.2cm}\\&
= &  \frac{\varepsilon}{2}+\frac{\lambda_i}{\lambda_1+\lambda_2}\left|
\text{meas}(\widetilde{\Omega}_\varepsilon^1)-\mu_1\text{meas}(\Omega)\right| \vspace{0.2cm}\\
& \le & \varepsilon , \end{array}$$
where we have used (\ref{meas12}).
$\hfill\Box$\medskip

We can now prove Theorem \ref{sufficientconditionforrelaxation}.\medskip

\noindent\textbf{Proof of Theorem \ref{sufficientconditionforrelaxation}.}
Let $\Omega$ be an open bounded
subset of $\mathbb{R}^n$, $\xi\in\operatorname*{int}K$ and let us denote by $u_{\xi}$ an affine map such that
$Du_\xi=\xi$. We want to construct a sequence
$u_\varepsilon\in Aff_{piec}(\overline{\Omega};\mathbb{R}^N)$ such that
$$\begin{array}{l} u_\varepsilon\in
u_\xi+W_0^{1,\infty}(\Omega;\mathbb{R}^N),\quad
Du_\varepsilon(x)\in E\cup \operatorname*{int}K,\ a.e.\ x\in\Omega,\vspace{0.2cm}\\
u_\varepsilon\stackrel{*}{\rightharpoonup}u_\xi\ \mbox{in}\
W^{1,\infty}(\Omega;\mathbb{R}^N),\quad
\displaystyle{\lim_{\varepsilon\rightarrow 0^+}\int_\Omega
\mathrm{dist}(Du_\varepsilon(x);E)\,dx =0}.
\end{array}$$

Fix $\varepsilon>0$. From condition $(ii)$, and since $\xi\in\operatorname*{int}K$, we have
$\xi\in K_\delta$ for $\delta\le\delta_0.$
Choose $0<\delta\le\delta_0$ such that $\delta<\varepsilon$ and
\begin{equation}\label{disteps}
\mathrm{dist}(\eta;E)\le\varepsilon,\ \forall\ \eta\in E_\delta.
\end{equation}
This is possible from condition $(i)$.
We then apply condition $(iii)$ to obtain $I=I(\delta)\in\mathbb{N},\
 (\lambda_i,\xi_i)_{1\le i\le I}$ with $\lambda_i>0$,   $\sum_{i=1}^I\lambda_i=1$,
$\xi_i\in\mathbb{R}^{N\times n}$  satisfying $(H_I(\operatorname*{int}K))$ and such that
\begin{equation}\label{sumdelta}
\displaystyle{\xi=\sum_{i=1}^I\lambda_i\xi_i\quad \text{and}\quad
\sum_{\substack{i=1\\ \xi_i\notin E_\delta}}^I\lambda_i<\delta.}
\end{equation}
By Lemma \ref{lemma} we now get $u_\varepsilon\in Aff_{piec}(\overline{\Omega};\mathbb{R}^N)$ and disjoint
open sets $\Omega_\varepsilon^i\subset\Omega$ such that, for $\varepsilon$ sufficiently small,
\begin{eqnarray}\label{meas}
& & u_\varepsilon\in u_\xi+W_0^{1,\infty}(\Omega;\mathbb{R}^N),\vspace{0.2cm} \nonumber \\
& & Du_\varepsilon(x)\in \mathrm{int}\,K,\ a.e.\ x\in\Omega,\quad
Du_\varepsilon(x)=\xi_i,\ a.e.\ x\in\Omega_\varepsilon^i,\ i=1,...,I,\vspace{0.2cm} \nonumber \\
& & ||u_\varepsilon -u_\xi||_{L^{\infty}}\le\varepsilon,\quad
|\text{meas}(\Omega_\varepsilon^i)-\lambda_i\text{meas}(\Omega)|\le\frac{\varepsilon}{I},\ i=1,...,I.
\end{eqnarray}
Since $K$ is bounded, up to a subsequence, we have
$u_\varepsilon\stackrel{*}{\rightharpoonup}u_\xi\ \mbox{in}\
W^{1,\infty}(\Omega;\mathbb{R}^N).$
We will finish the proof by verifying that
$$\displaystyle{\lim_{\varepsilon\rightarrow 0^+}\int_\Omega
\mathrm{dist}(Du_\varepsilon(x);E)\,dx =0.}$$
Since $E$ and $K$ are bounded there exists a positive constant $c$ such that
\begin{equation}\label{distc}
\mathrm{dist}(\eta;E)\le c,\ \forall\ \eta\in \operatorname*{int}K.
\end{equation}
Then, using (\ref{disteps}), (\ref{distc}), (\ref{meas}) and (\ref{sumdelta}), we obtain the
following estimates,
$$\begin{array}{l}\displaystyle{\int_\Omega
\mathrm{dist}(Du_\varepsilon(x);E)\,dx}  = \vspace{0.2cm}\\ =
\displaystyle{\sum_{\substack{i=1\\ \xi_i\in
E_\delta}}^I\int_{\Omega^i_\varepsilon} \mathrm{dist}(\xi_i;E)\,dx
+\sum_{\substack{i=1\\ \xi_i\notin
E_\delta}}^I\int_{\Omega^i_\varepsilon}
\mathrm{dist}(\xi_i;E)\,dx+}
\vspace{0.2cm}\\ \qquad +
\displaystyle{\int_{\Omega\setminus\left(\cup_{i=1}^I\Omega^i_\varepsilon\right)}
\mathrm{dist}(Du_\varepsilon(x);E)\,dx}\vspace{0.2cm}\\ \le \displaystyle{\varepsilon\,
\text{meas}(\Omega)+c \sum_{\substack{i=1\\ \xi_i\notin
E_\delta}}^I\text{meas}(\Omega^i_\varepsilon)+c\,\text{meas}\left(\Omega\setminus\left
(\cup_{i=1}^I\Omega^i_\varepsilon\right)\right)}\vspace{0.2cm}\\ \le
\displaystyle{\varepsilon\,\text{meas}(\Omega)+c \sum_{\substack{i=1\\ \xi_i\notin
E_\delta}}^I\left(\frac{\varepsilon}{I}+\lambda_i\text{meas}(\Omega)\right)+c\,\varepsilon}\vspace{0.2cm}\\
\le  \displaystyle{\varepsilon\,\text{meas}(\Omega)+c\,\varepsilon
+c\,\varepsilon\,\text{meas}(\Omega)+c\,\varepsilon}.
\end{array}$$
This completes the proof.$\hfill\Box$\medskip

As already mentioned, the characterization of $\operatorname*{Rco}_f E$ obtained in
Section \ref{sectionreview} entails a similar condition to condition $(iii)$ of
Theorem \ref{sufficientconditionforrelaxation} under the approximation property assumption.
This is stated in the next proposition.

\begin{proposition}\label{approximpliesC}Let $E\subset\mathbb{R}^{N\times n }$ be a bounded set and
for $\delta>0$ let $E_\delta$ be compact sets such that
$\operatorname*{Rco}_fE_\delta\subset\mathrm{int}\operatorname*{Rco}_fE$ and\smallskip

(i) $\forall\ \varepsilon>0\ \exists\ \delta_0>0:\ \mathrm{dist}(\eta;E)\le \varepsilon, \forall\ \eta\in
E_\delta,\ \delta\in(0,\delta_0]$;\smallskip

(ii) $\eta\in \mathrm{int}\operatorname*{Rco}_fE\ \Rightarrow\ \exists\ \delta_0>0:\ \eta\in
\operatorname*{Rco}_fE_\delta,\  \forall\ \delta\in(0,\delta_0]$.\smallskip

\noindent Then the following condition is satisfied:

$\forall\ \delta>0\ \forall\ \xi\in \operatorname*{Rco}_fE_\delta\ \exists\ I\in\mathbb{N},\ \exists\
(\lambda_i,\xi_i)_{1\le i\le I} \text{ with } \lambda_i>0,\     \displaystyle\sum_{i=1}^I\lambda_i=1,$

$\xi_i\in\mathbb{R}^{N\times n}, \text{ satisfying } (H_I(\mathrm{int}\,\operatorname*{Rco}_fE))\text{ and }
\displaystyle{\xi=\sum_{i=1}^I\lambda_i\xi_i,\ \sum_{\substack{i=1\\ \xi_i\notin
B_\delta(E)}}^I\lambda_i<\delta}.$
\end{proposition}

\begin{proof}
Let $\delta>0$ and $\xi\in \operatorname*{Rco}_fE_\delta$. Since
$\operatorname*{Rco}_fE_\delta\subset\mathrm{int}\operatorname*{Rco}_fE$,
$\xi\in\mathrm{int}\operatorname*{Rco}_fE$ so, using $(ii)$, we conclude that $\xi\in
\operatorname*{Rco}_fE_{\mu}$ for all $\mu\le \mu_1$. Thus, by the characterization of
$\operatorname*{Rco}_fE_\mu$ stated in Theorem \ref{Rcofcharact} with $U=\mathrm{int}\operatorname*{Rco}_fE$,
for every $\mu\le \mu_1$ and for every $\varepsilon>0$, there exist $I\in \mathbb{N}$ and
$(\lambda_i,\eta_i)_{i=1,...,I}$ with  $\lambda_i>0,\ \sum_{i=1}^I\lambda_i=1$ satisfying $(H_I(U))$ and
such that
\begin{equation}\label{firstassertion}
\xi=\sum_{i=1}^I\lambda_i\eta_i,\qquad
\sum_{\substack{i=1\\ \eta_i\notin B_\varepsilon(E_{\mu})}}^I\lambda_i<\varepsilon.\end{equation}

On the other hand, by $(i)$, for every $\gamma>0$ there exists $\mu_2>0$ such that
\begin{equation}\label{secondassertion}
\mathrm{dist}(\eta;E)\le \gamma,\ \forall\ \eta\in E_\mu\text{ with }\mu\le \mu_2.
\end{equation}

Choosing in the previous conditions
$\gamma=\frac{\delta}{3},\ \mu_0\le \min\{\mu_1,\mu_2\},\ \varepsilon=\frac{\delta}{3}$
we conclude, by (\ref{firstassertion}), that there exist $I\in \mathbb{N}$ and
$(\lambda_i,\eta_i)_{i=1,...,I}$ with  $\lambda_i>0,\ \sum_{i=1}^I\lambda_i=1$ satisfying $(H_I(U))$
and such that
$$\xi=\sum_{i=1}^I\lambda_i\eta_i,\qquad
\sum_{\substack{i=1\\ \eta_i\notin B_{\delta/3}(E_{\mu_0})}}^I\lambda_i<\frac{\delta}{3}.$$

To obtain the desired condition, we observe that
$$\sum_{\substack{i=1\\ \eta_i\notin B_\delta(E)}}^I\lambda_i\le
\sum_{\substack{i=1\\ \eta_i\notin B_{2\delta/3}(E_{\mu_0})}}^I\lambda_i < \frac{2\delta}{3} < \delta.$$

Indeed, if $\eta_i\notin B_\delta(E)$ then $\mathrm{dist}(\eta_i; E)\ge \delta.$ From
(\ref{secondassertion}), $\mathrm{dist}(\eta;E)\le\frac{\delta}{3},\ \forall\ \eta\in E_{\mu_0}$.
This means that $E_{\mu_0}\subset B_{\delta/3}(E)$ and thus
$\mathrm{dist}(\eta_i;E_{\mu_0})\ge \frac{2\delta}{3}.$
\end{proof}

\section{Applications}\label{sectionexamples}

We will now recall some properties of isotropic sets and investigate similar properties when a
restriction on the sign of the determinant is considered. These results will be useful in the study of
some differential inclusions related with this type of sets which we present in
subsections \ref{section Isotropic Differential Inclusion} and
\ref{Differential inclusions for some SO(n) invariant sets}.

We start by giving the precise definition of isotropic set.

\begin{definition}
Let $E$ be a subset of $\mathbb{R}^{n\times n}$. We say $E$ is isotropic if
$RES\subseteq E$ for every $R, S$ in the orthogonal group
$\mathcal{O}(n)$.
\end{definition}

Isotropic sets can be easily described by means of the singular values of its matrices. Indeed, let
$0 \leq \lambda_1(\xi)\leq \cdots \leq \lambda_n(\xi)$ denote the singular values of the matrix $\xi$, that
is, the eigenvalues of the matrix $\sqrt{\xi \xi^t}$,
then the isotropic sets $E$ of $\mathbb{R}^{n\times n}$ are those which can be written in the form
\begin{equation}\label{isotropic}
E=\{\xi \in \mathbb{R}^{n\times n}: (\lambda_1(\xi),\cdots, \lambda_n(\xi))\in \Lambda_E\}\,,
\end{equation}
where $\Lambda_E$ is a set contained in
$\{(x_1, \cdots,x_n)\in \mathbb{R}^n: 0\leq x_1 \leq \cdots \leq x_n\}$.
This is a consequence of some properties of the singular values that we recall next.

The following decomposition holds (see \cite{HJ}): for every matrix $\xi\in\mathbb{R}^{n\times n}$
there exist $R,S \in \mathcal{O}(n)$ such that
\begin{equation}\label{matrix decomposition}
\xi= R \, {\rm diag}(\lambda_1(\xi), \cdots, \lambda_n(\xi)) S=R\left(\begin{array}{ccc}
\lambda_1(\xi) &  & \\
 & \ddots & \\
 & & \lambda_n(\xi)
\end{array}\right)S
\end{equation}
and, for every $\xi \in \mathbb{R}^{n \times n}$, $R, S \in \mathcal{O}(n)$
$$\lambda_i(\xi) = \lambda_i(R\xi S).$$

Moreover, one has
$$
\begin{array}{c}
\displaystyle \prod_{i=1}^n\lambda_i(\xi)=|\det \xi|
\qquad\text{and}\qquad
\displaystyle \sum_{i=1}^n(\lambda_i(\xi))^2=|\xi|^2.
\end{array}
$$
In particular, in the $2 \times 2$ case,
$\lambda_1$ and $\lambda_2$ are given by
$$
\begin{array}{l}
\displaystyle \lambda_1(\xi)=\frac 12
\left[\sqrt{|\xi|^2+2
|\det \xi|}-\sqrt{|\xi|^2-2 |\det \xi|}\right]
\vspace{0.15cm}
\\
\displaystyle \lambda_2(\xi)=\frac 12
\left[\sqrt{|\xi|^2+2
|\det \xi|}+\sqrt{|\xi|^2-2 |\det \xi|}\right].
\end{array}
$$

The functions $\lambda_i$ are continuous,
$\displaystyle \xi \to \prod_{i=k}^n\lambda_i(\xi)$ is
polyconvex for any $1 \leq k \leq n$ and $\lambda_n$ is a norm.
From this, clearly if the set $\Lambda_E$ in (\ref{isotropic}) is compact (respectively, open) then
$E$ is also compact (respectively, open).
On the other hand, if $E$ is compact the set $\Lambda_E$ can be taken to be compact and if
$E$ is open (\ref{isotropic}) holds for an open set $\Lambda_E \subset \mathbb{R}^n$.

In this section we will also be interested in sets of the form
\begin{equation}\label{isotropic with restriction}
E=\left\{\xi \in \mathbb{R}^{n\times n}: (\lambda_1(\xi),\cdots, \lambda_n(\xi))\in \Lambda_E,\ \det\xi\ge 0
\right\},
\end{equation}
where, as before, $\Lambda_E$ is a set contained in
$\{(x_1, \cdots,x_n)\in \mathbb{R}^n: 0\leq x_1 \leq \cdots \leq x_n\}$. We observe that these are not
isotropic sets, but just a class of $\mathcal{SO}(n)$ invariant sets, where $\mathcal{SO}(n)$ denotes
the special orthogonal group.

\begin{theorem}\label{Ercisotropic}
If $E \subseteq \mathbb{R}^{n\times n}$ has the form (\ref{isotropic with restriction}) for some compact set
$\Lambda_E$, then $\operatorname*{Rco}_f E$ has the same form, with $\Lambda_{\operatorname*{Rco}_f E}$
also compact.
\end{theorem}

\begin{proof}
Let $E\subseteq \mathbb{R}^{n\times n}$ be a set of the form (\ref{isotropic with restriction}) for some
compact set $\Lambda_E$. Then it is trivial to conclude that $\operatorname*{Rco}_f E$ is compact and
$\operatorname*{Rco}_f E\subseteq\{\xi\in\mathbb{R}^{n\times n}:\ \det \xi\ge 0\},$
this follows from the fact that $\xi \to -\det \xi$ is rank one convex.
We will show that
$$ {\operatorname*{Rco}}_f E = \left\{\xi \in \mathbb{R}^{n\times n}:\
(\lambda_1(\xi),\cdots, \lambda_n(\xi)) \in
\Lambda_{\operatorname*{Rco}_f E}, \, \det \xi\ge 0\right\}$$
where
$$ \Lambda_{\operatorname*{Rco}_f E} = \left\{x \in \mathbb{R}^n : x =
(\lambda_1(\xi),\cdots, \lambda_n(\xi))
\text{ for some } \xi \in {\operatorname*{Rco}}_f E \right\}.$$
Notice that, in particular, $\Lambda_{\operatorname*{Rco}_f E}$ is compact.
To achieve the desired representation of $\operatorname*{Rco}_f E$ we only need to show that if
$\xi\notin \operatorname*{Rco}_f E$ then for every
$R,S\in \mathcal{SO}(n)$ one has $R\xi S\notin \operatorname*{Rco}_f E$.
Let $\xi \notin \operatorname*{Rco}_f E$,
then there exists a rank one convex function
$f:\mathbb{R}^{n\times n}\to \mathbb{R}$ such that
$f|_{E}\leq 0$ and $f(\xi)>0.$
Let $R,S \in \mathcal{SO}(n)$ and define
$f_1(\eta):=f(R^{-1}\eta S^{-1}).$
Then $f_1$ is rank one convex
and for all $\eta \in E$,
$f_1(\eta)=f(R^{-1}\eta S^{-1})\leq 0,$
as $R^{-1}\eta S^{-1} \in E$. However $f_1(R\xi S)=f(\xi)>0$ and so
$R\xi S$ doesn't belong to $\operatorname*{Rco}_f E$.
\end{proof}

\subsection{Non-affine map with a finite number of gradients without rank one connections}
\label{KirchheimExample}

In \cite{Kirchheim}, Kirchheim proved the existence of non-affine maps with a finite number of values for the
gradient but
without rank one connections between them (see also the result obtained by Kirchheim and Preiss,
cf. \cite[Corollary 4.40]{Kirchheim-notes}, where a non-affine map whose gradient takes five possible
values not rank one connected was
constructed). Kirchheim's result is the following.

\begin{theorem}\label{teoexitngradientes} Let $N,n\ge 2$, $m \in \mathbb{N}$ and
$\Omega\subset\mathbb{R}^n$ be a bounded open set. Then there is a set
$E=\{\xi_1,...,\xi_m\}\subset\mathbb{R}^{N\times n}$ such that
 $$\operatorname*{rank}(\xi_i-\xi_j)=\min\{N,n\},\text{ if }i\neq j$$ and there are $\xi\notin E$ and $u\in
u_\xi+W_0^{1,\infty}(\Omega;\mathbb{R}^N)$ such that
$$Du(x)\in E,\ a.e.\ x\in\Omega,$$ where $u_\xi$ represents a map such that $Du_\xi=\xi$.
\end{theorem}

This theorem was obtained thanks to an abstract result also due to Kirchheim (cf. \cite[Theorem
5]{Kirchheim}). What we want to show in this section is that the same result can also be achieved by the Baire
categories method, cf. Theorem \ref{abstract existence theorem}. Evidently, since, as described in the
statement of the
theorem, the elements of the set $E$ are not rank one connected, the gradient of the affine boundary data
$\xi$ does not belong to $\operatorname*{Rco}E=E$. Therefore, to prove the relaxation property required to
apply Theorem \ref{abstract existence theorem}, one cannot use the usual approximation
property (cf. Definition \ref{Approximation property}).
However, as we will see, this difficulty can be overcome
by means of Corollary \ref{corol2}. Indeed, the set $E$ constructed by
Kirchheim is such
that the gradient of the boundary data $\xi$ belongs to the interior of
$\operatorname*{Rco}\nolimits_{f} E$ (cf. Remark \ref{KsubsetRcofE}).

We recall in the following lemma the properties of the set $E$ constructed by Kirchheim. For the construction
of the set we refer again to \cite{Kirchheim}.

\begin{lemma}\label{lemaKirchheim} Let $N,n\ge 2$ and denote by $B_{\frac{1}{2}}(0)$ the open ball of
$\mathbb{R}^{N\times n}$ centered at $0$ and with radius $\frac{1}{2}$. Then there exists a set
$E=\{\xi_1,...,\xi_m\}\subset\mathbb{R}^{N\times n}$, $m\in\mathbb{N}$, such that
$$\operatorname*{rank}(\xi_i-\xi_j)=\min\{N,n\},\text{ if }i\neq j$$
and $\mathrm{dist}(\xi;B_{\frac{1}{2}}(0))> 0,$ for every
$\xi\in E$. Moreover, for every $\xi\in E$ there exists
$\mathcal{M}_\xi\subset \mathbb{R}^{N\times n}$ such that
\begin{itemize}
\item[i)] $\mathcal{M}_\xi\subset\xi+\{\mu\in\mathbb{R}^{N\times n}
:\ \operatorname*{rank}\mu=1\}$
\item[ii)]
$\mathcal{M}_\xi\subset B_{\frac{1}{2}}(0),\ \# \mathcal{M}_\xi<4Nn$
\item [iii)] $\partial B_{\frac{1}{2}}(0)\subset
\bigcup_{\xi\in E}\operatorname*{int}(\operatorname*{co}(\{\xi\}\cup\mathcal{M}_\xi)).$
\end{itemize}
\end{lemma}
\medskip

We now prove Theorem \ref{teoexitngradientes}, for the set $E$ considered in the previous lemma,
using the Baire categories method.
This proof was already obtained by Ribeiro in \cite{Ribeiro-thesis} but we include it here for
the sake of completeness.
\medskip

\noindent{\bf Proof of Theorem \ref{teoexitngradientes}.}
We consider a set $E$ with the properties of the above lemma and, with the same notations of the lemma, we
define
$$\displaystyle{K=B_{\frac{1}{2}}(0)\cup\bigcup_{\xi\in E}\operatorname*{int}(\operatorname*{co}(\{\xi\}
\cup\mathcal{M}_\xi))}.$$
Notice that $K$ is an open and bounded set. Since $K\setminus E$ is non empty (for instance, it
contains $0$) let $\overline{\xi} \in K\setminus E$. We will show that $K$ has the relaxation
property with respect to $E$ by applying Corollary \ref{corol2}.
Once this is proved, we conclude by Theorem \ref{abstract existence theorem} that there exists
$u\in u_{\overline{\xi}}+W_0^{1,\infty}(\Omega;\mathbb{R}^N)$ such that $Du\in E$ with
$\overline{\xi}\notin E$.

We only need to ensure condition $(H)$ of Corollary \ref{corol2}. Let $\delta>0$ and
$\eta\in K\setminus B_\delta(E)$. If
$\eta\in B_{\frac{1}{2}}(0)$, by definition of $K$ and condition $iii)$ of Lemma \ref{lemaKirchheim}
one can easily reduce to the case
$\eta\in\operatorname*{int}(\operatorname*{co}(\{\xi\}\cup\mathcal{M}_\xi))\setminus
B_{\frac{1}{2}}(0)$, with $\xi\in E$, moving along any rank one direction.

Consider now the case where
$\eta \in \operatorname*{int}(\operatorname*{co}(\{\xi\}\cup\mathcal{M}_\xi))
\setminus B_{\frac{1}{2}}(0)$, for some $\xi\in E$.
In this case we can write
$$\eta=\sum_{j=1}^{k}\lambda_j\mu_j+\left(1-\sum_{j=1}^{k}\lambda_j\right)\xi,$$
for some $\lambda_j\in(0,1)$ such that
$\displaystyle{\sum_{j=1}^{k}\lambda_j<1}$ and for some $\mu_j\in\mathcal{M}_\xi$. By condition $ii)$ of
Lemma~\ref{lemaKirchheim} we have $k< 4Nn$.

Let $j^{\ast}\in\{1,...,k\}$ be such that
$\lambda_{j^{\ast}}|\xi-\mu_{j^{\ast}}|=\max_{1\le j\le k}\lambda_j|\xi-\mu_j|$
and consider the rank one direction $\xi-\mu_{j^{\ast}}$. We will show that it is possible to find
$c>0$, independent of $\eta$, such that $\eta_1=c(\xi-\mu_{j^{\ast}})$
satisfies $[\eta-\eta_1,\eta+\eta_1]\subset\operatorname*{int}K$.
In particular, since $\mathrm{dist}(\xi;B_{\frac{1}{2}}(0))> 0,$ for every $\xi\in E$, and
$\mu_j \in B_{\frac{1}{2}}(0)$ by condition $ii)$ of Lemma~\ref{lemaKirchheim}, it follows that,
for some $C>0$ independent of $\eta$ and $\xi$,
\begin{equation}\label{eta1}
|\eta_1| \ge C.
\end{equation}

To find the constant $c$ we proceed in the following way. We notice that, for
$$|t|<\min\left\{\lambda_{j^{\ast}}, 1-\displaystyle{\sum_{j=1}^{k}\lambda_j}\right\},$$
$$\eta+t(\xi-\mu_{j^{\ast}})=\sum_{\substack{j=1\\ j\neq
j^{\ast}}}^k\lambda_j\mu_j+(\lambda_{j^{\ast}}-t)\mu_{j^{\ast}}+
\left(1-\sum_{j=1}^{k}\lambda_j+t\right)\xi\in
\operatorname*{int}(\operatorname*{co}(\{\xi\}\cup\mathcal{M}_\xi)).$$
Thus we only need to show that
$$\min\left\{\lambda_{j^{\ast}},
1-\displaystyle{\sum_{j=1}^{k}\lambda_j}\right\}\ge c>0.$$
The estimate for $\lambda_{j^{\ast}}$ follows from $$\delta<|\eta-\xi|\le
\sum_{j=1}^{k}\lambda_j|\mu_j-\xi|\le 4Nn
\lambda_{j^{\ast}}|\mu_{j^\ast}-\xi|\le 4Nn
\lambda_{j^{\ast}}\max_{\substack{\xi\in
E\\ \mu\in\mathcal{M}_\xi}}|\mu-\xi|.$$
On the other hand, since
\begin{eqnarray*}
\frac{1}{2} & \le & \left| \eta\right|\le
\sum_{j=1}^{k}\lambda_j\left|\mu_j\right|+\left(1-\sum_{j=1}^{k}\lambda_j\right)\left|\xi\right|\le \\
&\le& \left(\max_{\substack{\xi\in
E\\ \mu\in\mathcal{M}_\xi}}\left|\mu\right|\right)
\sum_{j=1}^{k}\lambda_j+\left(1-\sum_{j=1}^{k}\lambda_j\right)\max_{\xi\in E}|\xi|,
\end{eqnarray*}
one gets
$$\displaystyle 1-\sum_{j=1}^{k}\lambda_j\ge\frac{\frac{1}{2}-\underset{\substack{\xi\in
E\\ \mu\in\mathcal{M}_\xi}}\max \left|\mu\right|}{\underset{\xi\in
E}\max \left|\xi\right|-\underset{\substack{\xi\in
E\\ \mu\in\mathcal{M}_\xi}}\max \left|\mu\right|}>0$$ as wished.

We argue that repeating the same reasoning with the matrix $\eta+\eta_1$ and so on,
after $i$ iterations of this procedure,
we obtain a sequence of rank one matrices
$\eta_1,...,\eta_i$ satisfying
$[\eta + \eta_1 + \cdots +\eta_{j-1} - \eta_j,\eta + \eta_1 + \cdots +\eta_{j-1} + \eta_j]
\subseteq \operatorname*{int}K$, $j = 1, \cdots, i$
and $\eta+\eta_1+...+\eta_i\in B_{\delta}(E)$ where $i\le L(\delta,E,K)$ is
independent of $\eta$.

Indeed, without loss of generality, assume that
$\left|\eta+\eta_1\right|\ge \left|\eta-\eta_1\right|$.
Then it follows that
$$\left|\eta+\eta_1\right|\ge C$$
since, by (\ref{eta1}),
$$2\left|\eta+\eta_1\right|^2\ge
\left|\eta+\eta_1\right|^2+\left|\eta-\eta_1\right|^2=2\left|\eta\right|^2+2\left|\eta_1\right|^2
\ge 2 C^2.$$

If $\eta+\eta_1\notin B_{\delta}(E)$ we obtain, as before, $\eta_2$ such that $\left|\eta_2\right|\ge C$ and
$[\eta+\eta_1-\eta_2,\eta+\eta_1+\eta_2]\subset \operatorname*{int} K=K$.
Again, assuming that $\left|\eta+\eta_1+\eta_2\right|\ge \left|\eta+\eta_1-\eta_2\right|$, one has
$$2\left|\eta+\eta_1+\eta_2\right|^2\ge
\left|\eta+\eta_1+\eta_2\right|^2+\left|\eta+\eta_1-\eta_2\right|^2
=2\left|\eta+\eta_1\right|^2+2\left|\eta_2\right|^2\ge 4 C^2.$$

After $i$ iterations of this procedure one gets
$\eta+\eta_1+...+\eta_i\in \operatorname*{int} K=K$ with
$$
\left|\eta+\eta_1+...+\eta_i\right|\ge \sqrt{i} C.
$$
Thus, $\left|\eta+\eta_1+...+\eta_i\right|\to+\infty$, as $i \to +\infty$, contradicting the
fact that $K$ is bounded. Therefore, for some $i$ bounded by a constant $L=L(\delta,E,K)$,
we must have $\eta+\eta_1+...+\eta_i\in K\cap B_{\delta}(E)$.

This concludes the proof of condition $(H)$ of Corollary \ref{corol2} and thus the proof.$\hfill\Box$

\subsection{Isotropic differential inclusion}\label{section Isotropic Differential Inclusion}

In this section we discuss the differential inclusion problem
\begin{equation}\label{differential inclusion isotropic}
\left\{
\begin{array}{ll}
D u(x) \in E, \,\, &\mbox{a.e. } x \in \Omega,
\\
u(x)=\varphi(x), \,\, & x \in \partial \Omega,
\end{array}
\right.
\end{equation}
where $\Omega$ is an open bounded subset of $\mathbb{R}^n$ and $E$ is a
compact subset of $\mathbb{R}^{n\times n}$ which is isotropic, that is to say, invariant under orthogonal
transformations.

We observe that a result due to Dacorogna and Marcellini \cite[Theorem 7.28]{Dac-Marc} provides a sufficient
condition for existence of solutions to this problem.  Indeed, denoting by
$\lambda_1(\xi)\le\lambda_2(\xi)\le\cdots\le \lambda_n(\xi)$ the
singular values of $\xi\in\mathbb{R}^{n\times n}$, if there exists $\eta\in E$ with
$\lambda_i(\eta)=\gamma_i>0$ and
$\varphi\in C^1_{piec}(\overline{\Omega};\mathbb{R}^n)$ is such that
$$D\varphi\in E\cup \left\{\xi\in\mathbb{R}^{n\times n}:\
\prod_{i=\tau}^n\lambda_i(\xi)<\prod_{i=\tau}^n\gamma_i,\
\tau=1,...,n\right\}$$ then (\ref{differential inclusion isotropic}) has
$W^{1,\infty}(\Omega; \mathbb{R}^n)$ solutions.

In the 2 dimensional case ($n=2$), a less restrictive condition can be obtained, although it is more
difficult to check
in concrete examples. This was studied by Croce \cite{Croce} (see also \cite{Crocethesis}), using the Baire
categories method
that we discussed in Section \ref{sectionsufficientconditions}, and by Barroso, Croce and Ribeiro \cite{BCR}
using the convex
integration method due to M{\"u}ller and {\v{S}}ver{\'a}k \cite
{Muller-Sverak1996,Muller-Sverakcounterexamples}. With both
methods, the result obtained was the following.

\begin{theorem}\label{existence theorem compact isotropic}
Let $E:=\{\xi \in \mathbb{R}^{2\times 2}: (\lambda_1(\xi),\lambda_2(\xi))\in \Lambda_E \},$
where
$\Lambda_E\subset\{(x,y) \in \mathbb{R}^2: 0<x \leq y\}$ is a compact set and let
$$K= \left\{\xi \in \mathbb{R}^{2\times 2}:\ f_{\theta}(\lambda_1(\xi),\lambda_2(\xi))<
\max\limits_{(a,b) \in \Lambda_E}f_{\theta}(a,b),
\,\forall\,\,\theta \in [0,\max\limits_{(a,b)\in \Lambda_E}b]\right\},$$
where $f_{\theta}(x,y) := xy + \theta(y-x).$ Then, if $\Omega\subset \mathbb{R}^2$ is a bounded open set and if
$\varphi \in C^1_{piec}(\overline{\Omega}, \mathbb{R}^2)$ is such that
$D \varphi \in E \cup K$ a.e. in $\Omega$, there exists
a map $u\in \varphi+ W^{1,\infty}_0(\Omega, \mathbb{R}^2)$ such that
$Du \in E$ a.e. in $\Omega$.
\end{theorem}

We notice that it turns out that
$K=\operatorname*{int}\operatorname*{Rco}E=\operatorname*{int}\operatorname*{Rco}_fE$. To
achieve the previous theorem two fundamental results due to Cardaliaguet and Tahraoui \cite{CT} were used. On
one hand, they
characterized the polyconvex hull of any set $E$ as in the theorem; based on their description,
Croce \cite{Croce} then showed that
$$\operatorname*{Pco}E=\left\{\xi \in \mathbb{R}^{2\times 2}:\ f_{\theta}(\lambda_1(\xi),\lambda_2(\xi))\le
\max\limits_{(a,b) \in \Lambda_E}f_{\theta}(a,b),
\,\forall\,\,\theta \in [0,\max\limits_{(a,b)\in \Lambda_E}b]\right\}.$$
On the other hand, Cardaliaguet and Tahraoui \cite{CT} showed that, in dimension 2, compact isotropic sets
which are rank one convex are also polyconvex (see also \cite{CDLMR}).
Since the hull $\operatorname*{Rco}_fE$  of a compact isotropic set $E$ is also
compact and isotropic (cf. Theorem \ref{Ercisotropic}) and since it is also rank one convex,
one immediately obtains that
$\operatorname*{Rco}_fE=\operatorname*{Pco}E$ and thus a characterization for $\operatorname*{Rco}_fE$. This
was fundamental
to study the differential inclusion by means of the convex integration method. Indeed, to prove the required
in-approximation
it is necessary to know the hull $\operatorname*{Rco}_fE$. In the prior work of Croce \cite{Croce}, where the
Baire categories method was used via the approximation property, the appropriate hull to consider was
$\operatorname*{Rco}E$. Contrary to the case of $\operatorname*{Rco}_fE$,
a characterization of $\operatorname*{Rco}E$ does not follow immediately from Cardaliaguet and
Tahraoui's results since, in general, $\operatorname*{Rco}E$ may not be compact. For this reason,
in \cite{Croce}, $\operatorname*{Rco}E$ had to be computed (and the conclusion was that it coincides
with $\operatorname*{Rco}_fE$).
However, thanks to the theory presented in Section \ref{sectionsufficientconditions}, the results of
Cardaliaguet and Tahraoui are sufficient to obtain
Theorem \ref{existence theorem compact isotropic} using the Baire categories method.
We proceed with a brief sketch of this proof
which is essentially the one given in \cite{Croce}, but with no need of computing $\operatorname*{Rco}E$.
\medskip

\noindent{\bf Proof of Theorem \ref{existence theorem compact isotropic}.}
We recall that by the results in \cite{CT} and \cite{Croce},
$$\operatorname*{Rco}\nolimits_{f}E=\left\{\xi \in \mathbb{R}^{2\times 2}:\
f_{\theta}(\lambda_1(\xi),\lambda_2(\xi))\le
\max\limits_{(a,b) \in \Lambda_E}f_{\theta}(a,b),
\,\forall\,\,\theta \in [0,\max\limits_{(a,b)\in \Lambda_E}b]\right\}$$ and
$$\operatorname*{int}\operatorname*{Rco}\nolimits_{f}E=\left\{\xi \in \mathbb{R}^{2\times 2}\hspace{-0,1cm}
:\hspace{-0,1cm}
f_{\theta}(\lambda_1(\xi),\lambda_2(\xi))< \hspace{-0,1cm} \max\limits_{(a,b) \in \Lambda_E}f_{\theta}(a,b),
\forall \theta \in [0,\max\limits_{(a,b)\in \Lambda_E}b]\right\}.$$
Thus, by Corollary \ref{corol1} and Theorem \ref{abstract existence theorem}, it is enough to construct
compact sets
$E_\delta$ such that
$\operatorname*{Rco}\nolimits_{f}E_\delta\subset\mathrm{int}\operatorname*{Rco}\nolimits_{f}E$ and
satisfying conditions $(i)$ and $(ii)$ of Theorem \ref{sufficientconditionforrelaxation} with
$K_\delta=\operatorname*{Rco}_f E_\delta$. In fact, this is the case if we consider
$$\displaystyle E_\delta=\bigcup_{(a,b)\in \Lambda_E}\left\{\xi\in\mathbb{R}^{2 \times 2} :\
(\lambda_1(\xi),\lambda_2(\xi))=(a-\delta,b-\delta)\right\},$$
for $\displaystyle 0\le \delta\le \min_{(a,b)\in\Lambda_E}\frac{a}{2}$.
We refer to \cite{Croce} for the details of the proof.$\hfill\Box$

\subsection{Differential inclusions for some $\mathcal{SO}(n)$ invariant sets}
\label{Differential inclusions for some SO(n) invariant sets}

We consider in this section the differential inclusion problem
\begin{equation}\label{differential inclusion isotropic with determinant restriction}
\left\{
\begin{array}{ll}
D u(x) \in E, \,\, &\mbox{a.e. } x \in \Omega,
\\
u(x)=\varphi(x), \,\, & x \in \partial \Omega,
\end{array}
\right.
\end{equation}
in the case $E$ has the form
$$E=\left\{\xi \in \mathbb{R}^{n\times n}: (\lambda_1(\xi),\cdots, \lambda_n(\xi))\in \Lambda_E,\ \det\xi>
0\right\},$$
with $\Lambda_E\subseteq\{(x_1, \cdots,x_n)\in \mathbb{R}^n: 0< x_1 \leq \cdots \leq x_n\}$. As already
observed, these are not isotropic sets, but just a class of $\mathcal{SO}(n)$ invariant sets.

For $n=2$, Cardaliaguet and Tahraoui \cite{CT2} defined the set $R(\Lambda_E)$
for any compact set $\Lambda_E\subset \{(x_1,x_2) \in \mathbb{R}^{2}: 0\leq x_1\leq x_2\}$ as the smallest compact subset of
$\{(x_1,x_2) \in \mathbb{R}^2: 0\leq x_1\leq x_2\}$ containing $\Lambda_E$ such that
$$
\{\xi \in \mathbb{R}^{2\times 2}: (\lambda_1(\xi),\lambda_2(\xi))\in R(\Lambda_E),\ \det\xi\geq 0\}
$$
is rank one convex.
This hull can be used to describe $\operatorname*{Rco}_f E$
(see Lemma \ref{lemma_equivalenza_inviluppi}).
The representation of this envelop is quite complicated and leads to some difficulty in dealing with it
in order to show existence of solutions to problem
(\ref{differential inclusion isotropic with determinant restriction}).
In Theorem \ref{existence isotropic with determinant constraint two points} we consider a particular set $E$
composed by matrices with two possible singular values and, using Cardaliaguet and Tahraoui's results, we
give a sufficient condition for existence, relating the gradient of the
boundary data and the hull $\operatorname*{Rco}_fE$.

For $n>2$, a representation of $\operatorname*{Rco}_fE$ is not available. Of course, if one wants to ensure
existence of solutions to (\ref{differential inclusion isotropic with determinant restriction}), one may not
need to know the entire hull. Moreover, we notice that in the applications it is more
convenient to have simpler conditions to check
than those describing the hull $\operatorname*{Rco}_fE$ obtained by Cardaliaguet and Tahraoui \cite{CT2}
for the 2 dimensional case. In
this sense, in Theorems \ref{existence isotropic with determinant constraint n=2} and
\ref{existence isotropic with determinant constraint n=3} we will establish sufficient conditions for
existence of solutions to problem (\ref{differential inclusion isotropic with determinant restriction})
for certain sets $E$ in dimension 2 and 3. Analogous results could be obtained in higher dimensions
however, due to the heavy notation already present in the 3 dimensional case, we have only considered these
two settings.

\subsubsection{Set of singular values consisting of two points}\label{section The case of two points}
In this section we are going to consider the case where
$$
E=\{\xi \in \mathbb{R}^{2\times 2}: (\lambda_1(\xi),\lambda_2(\xi))\in \Lambda_E,\ \det\xi> 0\}
$$
with $\Lambda_E=\{(a_1,a_2),(b_1,b_2)\}$ and $0<a_1<b_1<a_2<b_2$.
We start by studying the set $\operatorname*{Rco}\nolimits_{f}E$. To this effect we will use the following characterization
of $R(\Lambda_E)$ obtained in \cite[Proposition 8.6, Theorem 7.1 and
Definition 1.1]{CT2}.

\begin{theorem}\label{thmCT2}
Let $\Lambda$ be a compact subset of $\{(x_1,x_2) \in \mathbb{R}^{2}: 0\leq x_1\leq x_2\}$ such that
$R(\Lambda)$
is connected. Then
$$
R(\Lambda)=\{(x_1,x_2)\in \mathbb{R}^{2} \hspace{-0,1cm}: \hspace{-0,07cm} 0\leq x_1\leq x_2, x_1\geq \alpha,
\sigma_3(x_1)\leq x_2\leq
\inf\{\sigma_1(x_1),\sigma_2(x_1)\}\}
$$
where
$\alpha=\inf\limits_{(x_1,x_2) \in \Lambda} x_1$,$$
\begin{array}{ll}
\displaystyle \sigma_1(x_1)=\inf_{(\theta,\gamma)\in \Sigma_1}f^1_{\theta,\gamma}(x_1), &
\displaystyle
f^1_{\theta,\gamma}(x_1)=\left\{
\begin{array}{ll}
\theta+\frac{\gamma - \theta^2}{\theta -x_1}, & x_1<\theta,
\\
+\infty, & \textnormal{otherwise},
\end{array}
\right.
\\
\displaystyle \sigma_2(x_1)=\inf_{(\theta,\gamma)\in \Sigma_2}f^2_{\theta,\gamma}(x_1), &
\displaystyle f^2_{\theta,\gamma}(x_1)=
\theta+\frac{\gamma - \theta^2}{\theta +x_1}\,,
\\
\displaystyle \sigma_3(x_1)=\sup_{(\theta,\gamma)\in \Sigma_3}f^3_{\theta,\gamma}(x_1), &
\displaystyle f^3_{\theta,\gamma}(x_1)=\left\{
\begin{array}{ll}
-\theta+\frac{\gamma - \theta^2}{x_1-\theta}, & x_1>\theta\,,
\\
+\infty, & \textnormal{otherwise}\,,
\end{array}
\right.
\end{array}
$$
$$
\Sigma_0=\{(\theta,\gamma) \in \mathbb{R}^2:\  \theta\ge 0,\ \gamma\ge\theta^2\},
$$
$$
\Sigma_1=\{(\theta,\gamma) \in \Sigma_0:\   x_2\leq f^1_{\theta,\gamma}(x_1),\ \forall\,(x_1,x_2) \in
\Lambda\},
$$
$$
\Sigma_2=\{(\theta,\gamma) \in \Sigma_0:\  x_2\leq f^2_{\theta,\gamma}(x_1),\ \forall\,(x_1,x_2) \in \Lambda\},
$$
$$
\Sigma_3=\{(\theta,\gamma) \in \Sigma_0:\  x_2\geq f^3_{\theta,\gamma}(x_1),\ \forall\,(x_1,x_2) \in \Lambda\}.
$$
Moreover,
there exists a convex function $h: \mathbb{R}^{2\times 2} \times \mathbb{R}\to \mathbb{R}$  such that
\begin{eqnarray*}
& & \big\{\xi \in \mathbb{R}^{2\times 2}: (\lambda_1(\xi),\lambda_2(\xi))\in R(\Lambda),\
\det \xi \ge 0\big\} \\
& & =\big\{\xi \in \mathbb{R}^{2\times 2}: h(\xi,\det \xi)\leq 0,\ \det \xi \ge 0\big\}.
\end{eqnarray*}
\end{theorem}\medskip

The following sufficient condition for $R(\Lambda)$ to be connected was also proven in
\cite[Proposition 8.4]{CT2}.

\begin{proposition}\label{connected}
Let $\Lambda$ be a compact subset of $\{(x_1,x_2) \in \mathbb{R}^{2}: 0\leq x_1\leq x_2\}$ and
assume that $\Lambda$ satisfies the following property:
if there exist $C_1$ and $C_2$, compact subsets of $\Lambda$, such that $C_1 \cap C_2 = \emptyset$,
$C_1 \cup C_2 = \Lambda$ and
$\underset{(x_1,x_2) \in C_1}\sup x_2 < \underset{(x_1,x_2) \in C_2}\inf x_1$,
then either $C_1 = \emptyset$ or $C_2 = \emptyset$.
Then $R(\Lambda)$ is connected.
\end{proposition}

If $E = \left\{\xi \in \mathbb{R}^{2\times 2}: (\lambda_1(\xi),\lambda_2(\xi))\in \Lambda_E,\ \det\xi\ge 0
\right\}$,
where $\Lambda_E$ is a compact subset of $\{(x_1,x_2) \in \mathbb{R}^2: 0\leq x_1\leq x_2\}$, then
$R(\Lambda_E)$ describes the hull $\operatorname*{Rco}\nolimits_{f}E$, as we prove in the following lemma.

\begin{lemma}\label{lemma_equivalenza_inviluppi}
Let $E=\{\xi \in \mathbb{R}^{2\times 2}: (\lambda_1(\xi),\lambda_2(\xi))\in \Lambda_E,\ \det\xi\ge 0\}$
where $\Lambda_E$ is a compact subset of $\{(x_1,x_2) \in \mathbb{R}^2: 0\leq x_1\leq x_2\}$ such that
$R(\Lambda_E)$ is connected. Then
$$
\operatorname*{Rco}\nolimits_{f}E=\{\xi \in \mathbb{R}^{2\times 2}:
(\lambda_1(\xi),\lambda_2(\xi))\in R(\Lambda_E),\ \det \xi \ge 0\}.
$$
\end{lemma}

\begin{proof}
We set
$$
\mathcal{E}=\{\xi \in \mathbb{R}^{2\times 2}: (\lambda_1(\xi),\lambda_2(\xi)) \in R(\Lambda_E),\
\det \xi\geq 0\}.
$$
Since $R(\Lambda_E)$  is connected, by Theorem \ref{thmCT2} there exists a convex function
$h: \mathbb{R}^{2\times 2}\times \mathbb{R}\to \mathbb{R}$
such that
$
\mathcal{E}=\{\xi \in \mathbb{R}^{2\times 2}: h(\xi, \det \xi)\leq 0,\ \det \xi\geq 0\}.
$
Note that  $h(\xi, \det \xi)\leq 0$ for every $\xi \in E$, since $E\subset \mathcal{E}$.
According to Definition \ref{hullsfinite} of $\operatorname*{Pco}\nolimits_{f}E$,  one has
$\operatorname*{Pco}\nolimits_{f}E \subseteq \mathcal{E}$.
This implies that $\operatorname*{Rco}\nolimits_{f}E\subseteq \mathcal{E}$.

On the other hand, $\operatorname*{Rco}\nolimits_{f}E$ is a compact and rank one convex set. By
Theorem \ref{Ercisotropic},
$$
\operatorname*{Rco}\nolimits_{f}E=\{\xi \in \mathbb{R}^{2\times 2}:\
(\lambda_1(\xi),\lambda_2(\xi))\in \tilde{\Lambda},\ \det \xi \geq 0\}
$$
for some compact subset $\tilde{\Lambda}$ of $\{(x_1,x_2) \in \mathbb{R}^2: 0\leq x_1\leq x_2\}$.
By definition of $R(\Lambda_E)$, $\tilde{\Lambda} \supseteq R(\Lambda_E)$ and thus
$\operatorname*{Rco}\nolimits_{f}E\supseteq \mathcal{E}$.
\end{proof}

Using the two previous results we can show the following formula for
$\operatorname*{Rco}\nolimits_{f}E$, for
the set  $E$ considered in this section.

\begin{proposition}\label{rcoftwopoints}
Let
$
E=\{\xi \in \mathbb{R}^{2\times 2}: (\lambda_1(\xi),\lambda_2(\xi))\in \Lambda_E,\ \det \xi>0\},
$
where $\Lambda_E=\{(a_1,a_2),(b_1,b_2)\}$, $0<a_1<b_1<a_2<b_2$, and define
$$\displaystyle \underline{\theta}=\frac{-a_1 a_2+b_1 b_2}{b_1+b_2-a_1-a_2}.$$
Then
\begin{equation}\label{formulaRcoftwopoints}
\operatorname*{Rco}\nolimits_{f}E=\{\xi \in \mathbb{R}^{2\times 2}: (\lambda_1(\xi),\lambda_2(\xi))
\in R(\Lambda_E),\ \det \xi> 0\}
\end{equation}
where
$R(\Lambda_E)$ is the set of points $(x_1,x_2) \in \mathbb{R}^2$ such that
\begin{equation}\label{formula_RLambda}
\left\{
\begin{array}{l}
a_1\leq x_1\leq x_2\leq b_2,\vspace{0.2cm}
\\
\displaystyle
x_2\geq \frac{a_1 a_2}{x_1},\vspace{0.2cm}
\\
\displaystyle
x_2\leq \frac{-a_1 a_2+\underline{\theta}(a_1+a_2)-\underline{\theta} x_1}{-x_1+\underline{\theta}},\quad \text{if}\ a_1\le x_1\le b_1\vspace{0.2cm}
\\
\displaystyle
x_2\leq\frac{b_1 b_2}{x_1}.
\end{array}
\right.
\end{equation}
Moreover,
$$
\mathrm{int}\operatorname*{Rco}\nolimits_{f}E=\{\xi \in \mathbb{R}^{2\times 2}:
(\lambda_1(\xi),\lambda_2(\xi))\in \mathrm{rel\,int}R(\Lambda_E),\ \det \xi> 0\},
$$
where
$\mathrm{rel\,int}R(\Lambda_E)$ is the relative interior of $R(\Lambda_E)$ with
respect to the set \newline
$\{(x_1,x_2) \in \mathbb{R}^2: 0 < x_1\leq x_2\}$ and is the
set of points $(x_1,x_2) \in \mathbb{R}^2$ such that
$$
\left\{
\begin{array}{l}
a_1<x_1\leq x_2<b_2,\vspace{0.2cm}
\\
\displaystyle
x_2>\frac{a_1 a_2}{x_1},\vspace{0.2cm}
\\
\displaystyle
x_2<\frac{-a_1 a_2+\underline{\theta}(a_1+a_2)-\underline{\theta} x_1}{-x_1+\underline{\theta}},\quad \text{if}\ a_1< x_1< b_1\vspace{0.2cm}
\\
\displaystyle
x_2<\frac{b_1 b_2}{x_1}.
\end{array}
\right.
$$
\end{proposition}
\begin{proof}
Notice that, since $a_1 > 0$ and $\lambda_1(\xi)\lambda_2(\xi) = |\det \xi|$,
our set $E$ satisfies the hypotheses of Lemma \ref{lemma_equivalenza_inviluppi} since,
by Proposition \ref{connected},
the set $R(\Lambda_E)$ is connected. Thus (\ref{formulaRcoftwopoints}) holds and to establish
(\ref{formula_RLambda})
we will write the inequalities
$\sigma_3(x_1)\leq x_2\leq \inf\{\sigma_1(x_1),\sigma_2(x_1)\}$, given by Theorem \ref{thmCT2},
for $x_1\geq a_1$, in a more explicit way.
Let us start by studying
$x_2\geq \sigma_3(x_1)$.
It is easy to see that
$$
\Sigma_3=\left\{(\theta,\gamma) \in \mathbb{R}^2: 0\leq \theta< a_1,\ \theta^2\leq \gamma\leq
a_1a_2-\theta(a_2-a_1)\right\},
$$
as $0<a_1<b_1<a_2<b_2$.
Therefore $x_2\geq \sigma_3(x_1)$ is equivalent to
$$
x_1 x_2-\theta(x_2-x_1)\geq a_1 a_2-\theta(a_2-a_1) \,,\,\,\,\theta \in [0,a_1)
$$
since $x_1>\theta$, that is,
\begin{equation}\label{primaiperbole}
x_2\geq \sup_{\theta \in [0,a_1)}\frac{a_1 a_2-\theta(a_2-a_1)-\theta x_1}{x_1-\theta}=\frac{a_1 a_2}{x_1}.
\end{equation}
In the same way, to study
$x_2\leq \sigma_2(x_1)$,
we remark that
$$
\Sigma_2=\left\{(\theta,\gamma)  \in \mathbb{R}^2: \theta\geq 0,\ \gamma\geq \max\{\theta^2, b_1
b_2+\theta(b_2-b_1)\}\right\}.
$$
Therefore $x_2\leq \sigma_2(x_1)$ is equivalent to
$$
x_1x_2+\theta(x_2-x_1)\leq \max\{\theta^2, b_1b_2+\theta(b_2-b_1)\}\,,\,\,\,\,\theta \geq 0\,,
$$
that is,
\begin{equation}\label{ypiupiccolodib2}
x_2\leq \min\left\{\frac{b_1b_2}{x_1}, b_2\right\}.
\end{equation}
We now analize the inequality $x_2\leq \sigma_1(x_1)$. It is easy to see that
$$
\Sigma_1=\{(\theta,\gamma)  \in \mathbb{R}^2: \theta\geq 0,\ \gamma\geq
\max\{\theta^2,  -a_1a_2+\theta(a_1+a_2),   -b_1b_2+\theta(b_1+b_2)\}\}\,.
$$
Therefore $x_2\leq \sigma_1(x_1)$ is equivalent to
$$
x_2\leq \frac{\gamma-\theta x_1}{\theta -x_1},\,\,\,\,\forall \, (\theta,\gamma) \in \Sigma_1: \theta>x_1\,,
$$
that is,
$$
-x_1 x_2+\theta(x_1+x_2)\leq \max\{\theta^2, -a_1 a_2+\theta(a_1+a_2), -b_1b_2+\theta(b_1+b_2)\},
\,\,\,\,\forall\,  \theta>x_1\,.
$$
By (\ref{ypiupiccolodib2}), $x_2\leq b_2$ and so
$$
-x_1x_2+\theta(x_1+x_2)\leq \max\{\theta^2, -a_1a_2+\theta(a_1+a_2),-b_1b_2+\theta(b_1+b_2)\},\,\,\theta
\in (x_1,b_2].
$$
Since $\theta>x_1\geq a_1$ and $0<a_1<b_1<a_2<b_2$,
one has
\begin{equation}\label{disu_ana}
x_2\leq \inf_{\theta \in (x_1, b_2]}\frac{\max\{-a_1a_2+\theta(a_1+a_2),-b_1b_2+\theta(b_1+b_2)\}-\theta
x_1}{-x_1+\theta}\,.
\end{equation}
We observe that
\begin{equation}\label{underlinetheta_position}
b_1<\underline{\theta}<a_2
\end{equation}
and
$$
\max\{-a_1a_2+\theta(a_1+a_2),-b_1b_2+\theta(b_1+b_2)\}
=\left\{
\begin{array}{ll}
\hspace{-0,25cm}-a_1a_2+\theta(a_1+a_2), & \hspace{-0,25cm}a_1\leq \theta\leq \underline{\theta}
\\
\hspace{-0,25cm}-b_1b_2+\theta(b_1+b_2), & \hspace{-0,25cm}\underline{\theta}\leq \theta\leq b_2\,.
\end{array}
\right.
$$
To study (\ref{disu_ana}) we distinguish the cases
$x_1\geq \underline{\theta}$
and
$x_1 < \underline{\theta}$.
In the first case, (\ref{disu_ana}) gives
$$
x_2\leq \inf_{\theta \in (x_1, b_2]}\frac{-b_1b_2+\theta(b_1+b_2)-\theta x_1}{-x_1+\theta}\,.
$$
Notice that the sign of the derivative of the function
$$\displaystyle g(\theta) := \frac{-b_1b_2+\theta(b_1+b_2)-\theta x_1}{-x_1+\theta}$$
does not depend on $\theta.$
Therefore we have
\begin{equation}\label{disu_ana_first_case}
x_2 \leq \min\Big\{g(b_2), \lim_{\theta \to x_1^+}g(\theta)\Big\} = g(b_2) = b_2\,,\,\,
\textnormal{if}\,\,x_1\geq \underline{\theta}\,.
\end{equation}
In the second case, (\ref{disu_ana}) yields
$$
x_2 \leq \min\Big\{
\inf_{\theta \in (x_1, \underline{\theta}]}\hspace{-0,1cm}\frac{-a_1a_2+\theta(a_1+a_2)-\theta x_1}
{-x_1+\theta},
\min_{\theta \in [\underline{\theta},b_2]}\hspace{-0,1cm}\frac{-b_1b_2+\theta(b_1+b_2)-\theta x_1}{-x_1+\theta}
\Big\}.
$$
As above, the sign of the derivatives of $g(\theta)$ and
$$\displaystyle f(\theta) :=  \frac{-a_1a_2+\theta(a_1+a_2)-\theta x_1}{-x_1+\theta}$$
does not depend on $\theta$, so we obtain
$$
x_2\leq \min\Big\{
f(\underline{\theta}), \lim_{\theta \to x_1^+}f(\theta),
g(\underline{\theta}), g(b_2)
\Big\} = \min\big\{
f(\underline{\theta}), b_2\big\},
$$
that is,
$$
x_2\leq \min\left\{
\frac{-a_1a_2+\underline{\theta}(a_1+a_2)-\underline{\theta} x_1}{-x_1+\underline{\theta}}, b_2
\right\},\,\,\textnormal{if}\,\,x_1 < \underline{\theta}\,.
$$
Therefore
{if} $x_1 < \underline{\theta}$
\begin{equation}\label{disu_ana_second_case}
x_2\leq
\left\{
\begin{array}{ll}
\displaystyle\frac{-a_1a_2+\underline{\theta}(a_1+a_2)-\underline{\theta} x_1}{-x_1+\underline{\theta}}, & x_1\leq b_1
\\
b_2, & b_1<x_1<\underline{\theta}\,.
\end{array}
\right.
\end{equation}
In conclusion, from (\ref{primaiperbole}), (\ref{ypiupiccolodib2}), (\ref{disu_ana_first_case}) and
(\ref{disu_ana_second_case}) we get
$$
\frac{a_1 a_2}{x_1}\leq x_2\leq \min\left\{b_2,\frac{b_1b_2}{x_1}\right\}\,,\,\,\textnormal{if}\,\,x_1>
b_1
$$
and
$$
\frac{a_1 a_2}{x_1} \leq x_2 \leq
\min\left\{b_2, \frac{b_1b_2}{x_1}, \frac{-a_1a_2+\underline{\theta}
(a_1+a_2)-\underline{\theta} x_1}{-x_1+\underline{\theta}}\right\}
,\,\,\textnormal{if}\,\,x_1 \leq b_1.
$$
By (\ref{underlinetheta_position}) and the fact that
$\displaystyle x_1\to\frac{-a_1a_2+\underline{\theta}(a_1+a_2)-
\underline{\theta} x_1}{-x_1+\underline{\theta}}$
passes through $(a_1,a_2)$ and $(b_1,b_2)$,
we get (\ref{formula_RLambda}).

\noindent
The formula of $\mathrm{int}\operatorname*{Rco}_fE$ is easy to obtain from the above representation.
\end{proof}

%\begin{center}
% \scalebox{0.8} {\psset{unit=1}\psset{algebraic}
%  \begin{pspicture}(-0.5,-0.5)(5,5)
%\psline{->}(-0.5,0)(4.5,0)
%\psline{->}(0,-0.5)(0,4.5)
%\psplot[linewidth=1pt]{1}{2}{(7-2.5*x)/(2.5-x)}
%\psplot[linewidth=1pt]{1}{1.73}{3/x}
%\psplot[linewidth=1pt]{2}{2.83}{8/x}
%\psline[linewidth=0.8pt](0,0)(4,4)
%\rput(4.5,0.3){${x_1}$}
%\rput(0.3,4.5){${x_2}$}
%\rput(2.4,1.6){$x_1=x_2$}
%\rput(3,4){$x_2=\frac{b_1b_2}{x_1}$}
%\rput(0.8,1.8){$x_2=\frac{a_1a_2}{x_1}$}
%\rput(-0,3.7){$x_2=\frac{-a_1a_2+\underline{\theta}(a_1+a_2)-\underline{\theta}x_1}{-x_1+\underline{\theta}}$}
%\end{pspicture}}
%\end{center}

We are now in position to prove an existence result for problem
(\ref{differential inclusion isotropic with determinant restriction}).
\begin{theorem}\label{existence isotropic with determinant constraint two points}
Let
$
E=\{\xi \in \mathbb{R}^{2\times 2}: (\lambda_1(\xi),\lambda_2(\xi))\in \Lambda_E,\ \det \xi> 0\}
$
where $\Lambda_E=\{(a_1,a_2),(b_1,b_2)\}$ and $0<a_1<b_1<a_2<b_2$.
Let $\Omega\subset \mathbb{R}^2$ be a bounded open set and let
$\varphi \in C^1_{piec}(\overline{\Omega}, \mathbb{R}^2)$ be such that
$D \varphi \in E \cup \mathrm{int}\operatorname*{Rco}_fE$ a.e. in $\Omega$. Then there exists
a map $u\in \varphi+ W^{1,\infty}_0(\Omega, \mathbb{R}^2)$ such that
$Du(x) \in E$ for a.e. $x$ in $\Omega$.
\end{theorem}
\begin{proof}
We will prove the result using Theorem \ref{abstract existence theorem} and Corollary \ref{corol1}.
Let $\delta>0$ be sufficiently small such that
$a_1+\delta<b_1<a_2<b_2-\delta$.
We define
$$
E_{\delta}=\{\xi \in \mathbb{R}^{2\times 2}:(\lambda_1(\xi),\lambda_2(\xi))\in \Lambda_{E_{\delta}},\
\det \xi> 0\},
$$
with $\Lambda_{E_{\delta}}=\{(a_1+\delta,a_2), (b_1, b_2-\delta)\}$. Observe that $E_\delta$ is a compact set.
By Proposition \ref{rcoftwopoints},
$\operatorname*{Rco}_fE_{\delta}$
is given by the  matrices $\xi \in \mathbb{R}^{2\times 2}$ with positive determinant such that
$(x_1,x_2)=(\lambda_1(\xi),\lambda_2(\xi))$ satisfies
\begin{equation}\label{previous_inequality_delta}
a_1 + \delta \leq x_1\leq x_2 \leq b_2 - \delta
\end{equation}
\begin{equation}\label{first_inequality_delta}
x_2\geq \frac{(a_1+\delta) a_2}{x_1}%\,,\,\,\textnormal{if}\,\, a_1+\delta
%\leq
%x_1 \leq {\sqrt{(a_1+\delta) a_2}}
\end{equation}
\begin{equation}\label{second_inequality_delta}
x_2 \leq
\frac{-(a_1 + \delta) a_2+\underline{\theta}_{\delta}(a_1 + \delta +a_2)-\underline{\theta}_{\delta} x_1}
{-x_1+\underline{\theta}_{\delta}}, \quad \text{if}\ a_1+\delta\le x_1\le b_1
%\,,\,\,\textnormal{if}\,\, a_1+\delta\leq x_1 \leq b_1
\end{equation}
\begin{equation}\label{third_inequality_delta}
x_2 \leq
\frac{b_1 (b_2-\delta)}{x_1}
%\,,\,\,\textnormal{if}\,\,
%b_1\leq x_1 \leq {\sqrt{b_1 (b_2-\delta)}}\,,
\end{equation}
where
$$
\underline{\theta}_{\delta}=\frac{-(a_1+\delta) a_2+b_1 (b_2-\delta)}{b_1+b_2-a_1-a_2-2\delta}.
$$
We are going to verify the hypotheses of Corollary \ref{corol1}.
We start by proving that
$\operatorname*{Rco}_fE_{\delta} \subset \mathrm{int}\operatorname*{Rco}_fE$.
Let $\xi \in \operatorname*{Rco}_fE_{\delta}$ and denote $(\lambda_1(\xi),\lambda_2(\xi))$ by $(x_1,x_2)$.
Since $(x_1,x_2)$ satisfies
%$a_1 + \delta \leq x_1\leq x_2 \leq b_2 - \delta$ and
inequalities (\ref{previous_inequality_delta}),
(\ref{first_inequality_delta}) and (\ref{third_inequality_delta}), it is clear that \newline
$a_1< x_1\leq x_2 < b_2$,
$\displaystyle
x_2> \frac{a_1 a_2}{x_1}$
and
$\displaystyle
x_2< \frac{b_1 b_2}{x_1}\,.
$
It remains to show that
\begin{equation}\label{inequality x_2}
x_2 <
\frac{-a_1 a_2+\underline{\theta}(a_1 +a_2)-\underline{\theta} x_1}
{-x_1+\underline{\theta}},\quad \text{if}\ a_1< x_1< b_1
\end{equation}
where $\displaystyle \underline{\theta}=\frac{-a_1 a_2+b_1 b_2}{b_1+b_2-a_1-a_2}$.
Since $x_1\geq a_1+\delta$, it suffices to show that if $x_1 \in [a_1+\delta,b_1)$, then
%$$
%\frac{-(a_1 + \delta) a_2+\underline{\theta}_{\delta}(a_1 + \delta +a_2)-\underline{\theta}_{\delta} x_1}
%{-x_1+\underline{\theta}_{\delta}}>\frac{b_1b_2}{x_1}
%$$
%By (\ref{second_inequality_delta}) the only thing to prove is that if $x_1 \in [a_1+\delta,b_1]$, then
$$
\frac{-(a_1+\delta) a_2+\underline{\theta}_{\delta}(a_1+\delta+a_2)-\underline{\theta}_{\delta}
x_1}{-x_1+\underline{\theta}_{\delta}}
<
\frac{-a_1 a_2+\underline{\theta}(a_1+a_2)-\underline{\theta} x_1}{-x_1+\underline{\theta}}.
$$
%where $\displaystyle \underline{\theta}=\frac{-a_1 a_2+b_1 b_2}{b_1+b_2-a_1-a_2}$.
As $b_1<\underline{\theta}<a_2$ and
$b_1<\underline{\theta}_{\delta}<a_2$,
the above inequality is equivalent to
\begin{equation}\label{parabola_retta}
(\underline{\theta}_{\delta}-\underline{\theta})(x_1-a_1)(x_1-a_2)<
\delta(-x_1+\underline{\theta})(a_2-\underline{\theta}_{\delta})\,,
\end{equation}
for $x_1 \in [a_1+\delta,b_1)$.
This inequality holds whenever $\underline{\theta}_{\delta}-\underline{\theta}>0$,
since the left hand side of (\ref{parabola_retta}) is negative and the right hand side is positive.
To show it also holds in the case $\underline{\theta}_{\delta}-\underline{\theta}<0$,
we notice that the graph of
$x_1 \to (\underline{\theta}_{\delta}-\underline{\theta})(x_1-a_1)(x_1-a_2)$
is a concave parabola passing through $(a_1,0),(a_2,0)$, whereas the graph of
$x_1 \to \delta(-x_1+\underline{\theta})(a_2-\underline{\theta}_{\delta})$ is a straight line with negative
slope passing through $(\underline{\theta},0)$. Therefore it is sufficient to prove (\ref{parabola_retta})
for $x_1=b_1$, that is,
$$
\left[\frac{-(a_1+\delta) a_2+b_1 (b_2-\delta)}{b_1+b_2-a_1-a_2-2\delta}-
\frac{-a_1 a_2+b_1 b_2}{b_1+b_2-a_1-a_2}\right]
(b_1-a_1)(b_1-a_2) <
$$
$$
<\delta\left[-b_1+\frac{-a_1 a_2+b_1 b_2}{b_1+b_2-a_1-a_2}\right]
\left[a_2-\frac{-(a_1+\delta) a_2+b_1 (b_2-\delta)}{b_1+b_2-a_1-a_2-2\delta}\right]\,.
$$
It is not difficult to see that the above inequality holds if and only if $\delta<b_1-a_1$
which is satisfied by the hypotheses on $\delta$.

The other conditions of Corollary \ref{corol1} are easy to check.
Indeed, any $\eta \in E_{\delta}$ can be written as $R\,\textnormal{diag}(a_1+\delta,a_2)\,S$ or
$R\,\textnormal{diag}(b_1,b_2-\delta)\,S$, for some $R,S \in \mathcal{SO}(2)$. In both cases
$\mathrm{dist}(\eta; E)\leq \delta
$.
This proves that for every $\varepsilon>0$ $\mathrm{dist}(\eta; E)\leq \varepsilon$
for every $\eta \in E_{\delta}$, with $\delta\leq \varepsilon$.

To prove the last condition of Corollary \ref{corol1},
let $\eta \in \mathrm{int}\operatorname*{Rco}_fE$.
Since  $a_1+\delta\to a_1$,
$b_2-\delta\to b_2$ and $\underline{\theta}_{\delta}\to \underline{\theta}$, as ${\delta}\to 0$,
we have that
$(\lambda_1(\eta), \lambda_2(\eta))$ satisfies (\ref{first_inequality_delta}),
(\ref{second_inequality_delta}) and
(\ref{third_inequality_delta}) for sufficiently small $\delta$, that is,
$\eta \in \textnormal{Rco}_f E_{\delta}$ for sufficiently small $\delta$.
\end{proof}

\subsubsection{Set of singular values containing a line segment}\label{section Set containing a line segment}

In this section we establish sufficient conditions for existence of solutions to problem
(\ref{differential inclusion isotropic with determinant restriction}) when $n=2$ and $n=3$. Our results rely on
the hypothesis that the set of singular values of the matrices in $E$ contains a line segment. We start by
considering the 2 dimensional case.

\begin{theorem}\label{existence isotropic with determinant constraint n=2}
Let
$E=\left\{\xi \in \mathbb{R}^{2\times 2}: (\lambda_1(\xi),\lambda_2(\xi))\in \Lambda_E,\
\det\xi> 0\right\},$ where
$\Lambda_E\subseteq\{(x_1, x_2)\in \mathbb{R}^2: 0<x_1 \leq x_2\}$, and assume that
$$\Gamma:=\{(a_1+t(b_1-a_1),a_2+t(b_2-a_2)):\ t\in[0,1]\}\subseteq \Lambda_E,$$
with $a_1<b_1$, $a_2<b_2$, and either $a_1<a_2$ or $b_1<b_2$.
Let $$K:=\left\{\xi \in \mathbb{R}^{2\times 2}: (\lambda_1(\xi),\lambda_2(\xi))\in \Lambda_K,\
\det\xi> 0\right\},$$ where
$$ \Lambda_K:=\bigcup_{(\alpha_1,\alpha_2)\in\mathrm{rel\,int}\operatorname*{\Gamma}}
\{(x_1,x_2)\in\mathbb{R}^2:\ x_1x_2=\alpha_1\alpha_2,\ 0<x_1\le x_2<\alpha_2\},$$
and $\mathrm{rel\,int}\operatorname*{\Gamma}$ is the relative interior of $\Gamma$ with respect to the line
joining $(a_1,a_2)$ and $(b_1,b_2)$, that is,
$$\mathrm{rel\,int}\operatorname*{\Gamma}:=\{(a_1+t(b_1-a_1),a_2+t(b_2-a_2)):\ t\in(0,1)\}.$$

Let $\Omega\subset \mathbb{R}^2$ be a bounded open set and let
$\varphi \in C^1_{piec}(\overline{\Omega}, \mathbb{R}^2)$ be such that
$D \varphi(x) \in E \cup K$ for a.e. $x$ in $\Omega$. Then there exists
a map $u\in \varphi+ W^{1,\infty}_0(\Omega, \mathbb{R}^2)$ such that
$Du(x) \in E$ for a.e.  $x$ in $\Omega$.
\end{theorem}

\begin{proof}
We start by noticing that, by the regularity of the boundary condition $\varphi$, one can assume that $E$ is
compact.
Let $\widetilde{E}$ be the subset of $E$ whose singular values lie in the segment joining $(a_1,a_2)$ and
$(b_1,b_2)$:
$$\widetilde{E}:=\left\{\xi \in \mathbb{R}^{2\times 2}: (\lambda_1(\xi),\lambda_2(\xi))\in
\Lambda_{\widetilde{E}},\ \det\xi> 0\right\},\text{ where }\Lambda_{\widetilde{E}}:=\Gamma.$$
Using Theorem \ref{sufficientconditionforrelaxation}, we will show that $K$ has the relaxation property with
respect to
$\widetilde{E}$ and thus with respect to $E$, since $\widetilde{E}\subseteq E$. The result will then follow
as an application of Theorem \ref{abstract existence theorem}.
Notice first that $\widetilde{E}$ is a bounded set since $\Lambda_{\widetilde{E}}$
is compact and $\lambda_2$ is a norm.

We will prove that the set $\Lambda_K$ is open in
$\{(x_1,x_2)\in\mathbb{R}^2:\ x_1\le x_2\}$. Let $y=(y_1,y_2)\in \Lambda_K$ and assume, by contradiction,
that, for each $\delta>0$, there exists $x=(x_1,x_2)\in B_\delta(y)$ with $x_1\le x_2$ and
$x \notin \Lambda_K$. Therefore, it is possible to construct a sequence $x^n=(x^n_1,x^n_2)$ converging to
$y$ with $x^n_1\le x^n_2$ and $x^n \notin \Lambda_K$.
Observe that, due to the hypotheses on $a_1,a_2,b_1,b_2$, the function
$\psi(t):=(a_1+t(b_1-a_1))(a_2+t(b_2-a_2))$ is
strictly increasing in $[0,1]$ and thus, a continuous bijection between $[a_1a_2,b_1b_2]$ and
$\Lambda_{\widetilde{E}}$
is defined. By definition of $\Lambda_K$, there exists
$(\alpha_1,\alpha_2)\in \mathrm{rel\,int}\operatorname*{\Gamma}$ such that
$y_1y_2=\alpha_1\alpha_2\in (a_1a_2,b_1b_2)$ with $y_2<\alpha_2$. By continuity of the product,
$\lim\limits_{n\to +\infty} x_1^nx_2^n=\alpha_1\alpha_2$. Hence, for sufficiently large $n\in\mathbb{N}$,
$x_1^nx_2^n\in(a_1a_2,b_1b_2)$ and thus, the existence of the bijection referred
to above implies that
$x^n_1x^n_2=\beta^n_1\beta^n_2$, for some
$(\beta^n_1,\beta^n_2)\in\mathrm{rel\,int}\operatorname*{\Gamma}$ with
$\lim\limits_{n\to +\infty}  \beta^n_1\beta^n_2=\alpha_1\alpha_2$. In particular, again by the continuity of the bijection
between $[a_1a_2,b_1b_2]$ and $\Lambda_{\widetilde{E}}$, $\lim\limits_{n\to +\infty}  (\beta^n_1,\beta^n_2)= (\alpha_1,\alpha_2)$.
Since, by hypothesis, $x^n \notin \Lambda_K$, then $x_2^n\ge \beta_2^n$ and passing to the limit, as $n\to +\infty$, we get
$y_2\ge \alpha_2$, which is a contradiction. So we conclude that $\Lambda_K$ is open.

Since the singular values are continuous functions and $\lambda_1\le\lambda_2$, it follows that $K$ is an open
set.
It remains thus to show that $K$ has the relaxation property with respect to $\widetilde{E}$, which will be
achieved through Theorem \ref{sufficientconditionforrelaxation}.

Before proceeding, we observe that $K\subset\operatorname*{int}\operatorname*{Rco}\widetilde{E}$. Indeed, it
follows from a result due to Dacorogna and Tanteri \cite{Dacorogna-Tanteri}
(see also \cite[Theorem 7.43]{DirectMethods}) that, for each
$(\alpha_1,\alpha_2)\in\mathrm{rel\,int}\operatorname*{\Lambda_{\widetilde{E}}}$,
\begin{eqnarray}\label{tanteri}
\operatorname*{Rco}\left\{\xi \in \mathbb{R}^{2\times 2}:
(\lambda_1(\xi),\lambda_2(\xi))=(\alpha_1,\alpha_2),\ \det\xi> 0\right\}= \vspace{0.2cm}\nonumber \\
=\left\{\xi \in \mathbb{R}^{2\times 2}:\ \det \xi=\alpha_1\alpha_2,\
\lambda_2(\xi)\le\alpha_2\right\}.
\end{eqnarray}
Therefore, it is clear that $K\subseteq\operatorname*{Rco}\widetilde{E}$ and since it is open, the desired
inclusion follows.
Moreover this inclusion implies that $K$ is bounded since $\widetilde{E}$ is bounded.

We also note that, since $\Lambda_{\widetilde{E}}\subset\{(x_1,x_2)\in\mathbb{R}^2:\ 0<x_1\le x_2\}$,
\begin{equation}\label{GammaKcharact}
\Lambda_K=\bigcup_{(\alpha_1,\alpha_2)\in\mathrm{rel\,int}\operatorname*{\Gamma}}
\left\{\left(\alpha_1+c,\frac{\alpha_1\alpha_2}{\alpha_1+c}\right)\in\mathbb{R}^2:
\ 0<c\le \sqrt{\alpha_1\alpha_2}-\alpha_1\right\}.
\end{equation}

Now we will prove the relaxation property introducing convenient approximating sets
$\widetilde{E}_\delta$ and $K_\delta$.
For sufficiently small $\delta$, let
$$c(\delta,\alpha_1,\alpha_2)=\min\{\delta,\sqrt{\alpha_1\alpha_2}-\alpha_1\},$$
$$\widetilde{E}_\delta:=\left\{\xi \in \mathbb{R}^{2\times 2}: (\lambda_1(\xi),\lambda_2(\xi))\in
\Lambda_{\widetilde{E}_\delta},\ \det\xi> 0\right\},$$
$$K_\delta:=\left\{\xi \in \mathbb{R}^{2\times 2}: (\lambda_1(\xi),\lambda_2(\xi))\in \Lambda_{K_\delta},\
\det\xi> 0\right\},$$
where
\begin{eqnarray*}\Lambda_{\widetilde{E}_\delta}:=\left(\bigcup_{(\alpha_1,\alpha_2)\in\mathrm{rel\,int}
\operatorname*{\Lambda_{\widetilde{E}}}}\left\{\left(\alpha_1+c(\delta,\alpha_1,\alpha_2),
\frac{\alpha_1\alpha_2}{\alpha_1+c(\delta,\alpha_1,\alpha_2)}\right)\right\}\right)\bigcap\vspace{0.2cm}\\
\bigcap\left\{(x_1,x_2)\in\mathbb{R}^2:\ 0<x_1\le x_2,\ a_1a_2+\delta\le x_1x_2\le b_1b_2-\delta\right\},
\end{eqnarray*}
$$
\Lambda_{K_\delta}:=\bigcup_{(\alpha_1,\alpha_2)\in\Lambda_{\widetilde{E}_\delta}}\{(x_1,x_2)\in\mathbb{R}^2:\
x_1x_2=\alpha_1\alpha_2,\ 0<x_1\le x_2\le\alpha_2\}.$$

We proceed with the proof of conditions $(i)$, $(ii)$ and $(iii)$ of
Theorem \ref{sufficientconditionforrelaxation}.

To prove condition $(i)$, by the matrix decomposition (\ref{matrix decomposition}), it is enough to show that,
for any given $\varepsilon>0$, there exists $\delta_0>0$ such that
$$\mathrm{dist}((x_1,x_2); \Lambda_{\widetilde{E}})\le\varepsilon,\
\forall\ (x_1,x_2)\in \Lambda_{\widetilde{E}_\delta},\ \delta\in(0,\delta_0].$$
The elements of $\Lambda_{\widetilde{E}_\delta}$ are of the form
$\left(\alpha_1+c(\delta,\alpha_1,\alpha_2),\frac{\alpha_1\alpha_2}
{\alpha_1+c(\delta,\alpha_1,\alpha_2)}\right)$
for some
$(\alpha_1,\alpha_2)\in\mathrm{rel\,int}\operatorname*{\Lambda_{\widetilde{E}}}\subset\Lambda_{\widetilde{E}}$.
Thus, choosing
$\delta_0=\frac{\varepsilon}{\sqrt{1+b_2^2/a_1^2}}$, we achieve the desired condition, since
$$\left|\left(\alpha_1+c(\delta,\alpha_1,\alpha_2),
\frac{\alpha_1\alpha_2}{\alpha_1+c(\delta,\alpha_1,\alpha_2)}\right)-(\alpha_1,\alpha_2)\right|\le
\delta\sqrt{1+b_2^2/a_1^2},$$
where we have used the fact that any
$(\alpha_1,\alpha_2)\in \mathrm{rel\,int}\operatorname*{\Lambda_{\widetilde{E}}}$
satisfies $\alpha_1\ge a_1$ and $\alpha_2\le b_2$ (this follows from the hypotheses on $a_1,a_2,b_1,b_2$).

Now we prove condition $(ii)$. Let $\eta\in \operatorname*{int}K=K$. It suffices to show that, for sufficiently
small $\delta$, $\lambda_1(\eta)\lambda_2(\eta)=\beta_1\beta_2$, for some
$(\beta_1,\beta_2)\in\Lambda_{\widetilde{E}_\delta}$ with $\lambda_2(\eta)\le \beta_2$.
By (\ref{GammaKcharact}),
$(\lambda_1(\eta),\lambda_2(\eta))=\left(\alpha_1+c,\frac{\alpha_1\alpha_2}{\alpha_1+c}\right)$ for some
$(\alpha_1,\alpha_2)\in\mathrm{rel\,int}\operatorname*{\Lambda_{\widetilde{E}}}$ and
$0<c\le \sqrt{\alpha_1\alpha_2}-\alpha_1$.
The monotonicity of the function $\psi$ introduced above, yields $\alpha_1\alpha_2\in(a_1a_2,b_1b_2)$.
Thus, for small $\delta$, one has $\alpha_1\alpha_2\in [a_1a_2+\delta, b_1b_2-\delta]$. Defining
$(\beta_1,\beta_2)=\left(\alpha_1+c(\delta,\alpha_1,\alpha_2),\frac{\alpha_1\alpha_2}
{\alpha_1+c(\delta,\alpha_1,\alpha_2)}\right)$, we observe that
$(\beta_1,\beta_2)\in \Lambda_{\widetilde{E}_\delta}.$
Finally, $\lambda_2(\eta)<\beta_2$ is equivalent to $c>c(\delta,\alpha_1,\alpha_2)$ and this is true for
sufficiently small
$\delta$, since $\displaystyle \lim_{\delta\rightarrow 0}c(\delta,\alpha_1,\alpha_2)=0$ and $c>0$.

It remains to prove condition $(iii)$. We notice that, using (\ref{tanteri}),
we conclude that any $\xi\in K_\delta$ belongs to
$$\operatorname*{Rco}\left\{\xi \in \mathbb{R}^{2\times 2}:
(\lambda_1(\xi),\lambda_2(\xi))=(\alpha_1,\alpha_2),\
\det\xi> 0\right\},$$
for some $(\alpha_1,\alpha_2)\in \Lambda_{\widetilde{E}_\delta}$. Since
$\Lambda_{\widetilde{E}_\delta}\subset\Lambda_K$,
this hull is a subset of $K=\operatorname*{int}K$ and condition $(iii)$ follows from (\ref{Rico}) (see also
Remark \ref{remark Rco characterization} - 3)).
\end{proof}\medskip

We consider next the 3 dimensional version of the previous result.

\begin{theorem}\label{existence isotropic with determinant constraint n=3}
Let
$E = \left\{\xi \in \mathbb{R}^{3\times 3}: (\lambda_1(\xi),\lambda_2(\xi),\lambda_3(\xi))\in \Lambda_E,\
\det\xi> 0\right\},$ \newline
where $\Lambda_E\subseteq\{(x_1, x_2, x_3)\in \mathbb{R}^3: 0<x_1 \leq x_2\leq x_3\},$ and assume that
$$\Gamma:=\{(a_1+t(b_1-a_1),a_2+t(b_2-a_2), a_3+t(b_3-a_3)):\ t\in[0,1]\}\subseteq \Lambda_E,$$
with $a_1<b_1$, $a_2<b_2$,
$a_3<b_3$ and either $a_1<a_2<a_3$ or $b_1<b_2<b_3$.
Let $$K:=\left\{\xi \in \mathbb{R}^{3\times 3}: (\lambda_1(\xi),\lambda_2(\xi),\lambda_3(\xi))\in \Lambda_K,\
\det\xi> 0\right\},$$ where
$$ \begin{array}{l}\displaystyle
\Lambda_K:=\bigcup_{(\alpha_1,\alpha_2,\alpha_3)\in\mathrm{rel\,int}\operatorname*{\Gamma}}
\left\{(x_1,x_2,x_3)\in\mathbb{R}^3:\ x_1 x_2 x_3=\alpha_1\alpha_2\alpha_3,\right.\vspace{0.1cm}\\
\left.\hspace{4cm} 0<x_1\le x_2\le x_3<\alpha_3,\ x_2 x_3<\alpha_2\alpha_3\right\},\end{array}$$
and $\mathrm{rel\,int}\operatorname*{\Gamma}$ is the relative interior of $\Gamma$ with respect to the line
joining $(a_1,a_2,a_3)$ and $(b_1,b_2,b_3)$, that is,
$$\mathrm{rel\,int}\operatorname*{\Gamma}:=\{(a_1+t(b_1-a_1),a_2+t(b_2-a_2), a_3+t(b_3-a_3)):\ t\in(0,1)\}.$$

Let $\Omega\subset \mathbb{R}^3$ be a bounded open set and let
$\varphi \in C^1_{piec}(\overline{\Omega}, \mathbb{R}^3)$ be such that
$D \varphi(x) \in E \cup K$ for a.e. $x$ in $\Omega$. Then there exists
a map $u\in \varphi+ W^{1,\infty}_0(\Omega, \mathbb{R}^3)$ such that
$Du(x) \in E$ for a.e. $x$ in $\Omega$.
\end{theorem}

\begin{proof} The proof of this result follows the lines of that of
Theorem \ref{existence isotropic with determinant constraint n=2}. Due to the heavy notation we won't present
it here in full detail. The reader can follow the proof of
Theorem \ref{existence isotropic with determinant constraint n=2} taking into account that in this case the
set $\Lambda_K$ can be written in the form
$$ \begin{array}{l}\displaystyle
\Lambda_K=\bigcup_{(\alpha_1,\alpha_2,\alpha_3)\in\mathrm{rel\,int}\operatorname*{\Gamma}}
\left\{\left(\alpha_1+c_1,\frac{\alpha_1\alpha_2}{\alpha_1+c_1}+c_2,
\frac{\alpha_1\alpha_2\alpha_3}{\alpha_1\alpha_2+c_2(\alpha_1+c_1)}\right)\in\mathbb{R}^3:
\right.\vspace{0.1cm}\\ \left.
\hspace{1.5cm} \displaystyle c_1>0,\ c_2\ge \alpha_1+c_1-\frac{\alpha_1\alpha_2}{\alpha_1+c_1},\
0<c_2\le\sqrt{\frac{\alpha_1\alpha_2\alpha_3}{\alpha_1+c_1}}-\frac{\alpha_1\alpha_2}{\alpha_1+c_1}
\right\}.\end{array}$$
The approximating sets $\widetilde{E}_\delta$ and $K_\delta$ can be defined, for
sufficiently small $\delta$, by
$$\widetilde{E}_\delta:=\left\{\xi \in \mathbb{R}^{3\times 3}:
(\lambda_1(\xi),\lambda_2(\xi),\lambda_3(\xi))\in
\Lambda_{\widetilde{E}_\delta},\ \det\xi> 0\right\},$$
$$K_\delta:=\left\{\xi \in \mathbb{R}^{3\times 3}: (\lambda_1(\xi),\lambda_2(\xi),\lambda_3(\xi))\in
\Lambda_{K_\delta},\
\det\xi> 0\right\},$$
where $\Lambda_{\widetilde{E}_\delta}$ is the set
\begin{eqnarray*}\Big(\bigcup_{(\alpha_1,\alpha_2,\alpha_3)\in\mathrm{rel\,int}\operatorname*{\Gamma}}
\Big\{\Big(\alpha_1+\delta,\frac{\alpha_1\alpha_2}{\alpha_1+\delta}+c(\delta,\alpha),
\frac{\alpha_1\alpha_2\alpha_3}{\alpha_1\alpha_2+c(\delta,\alpha)(\alpha_1+\delta)}\Big)\Big\}\Big)
\vspace{0.2cm}\\
\hspace{-0,2cm}\bigcap\left\{(x_1,x_2,x_3)\in\mathbb{R}^3:\hspace{-0,2cm}\ 0<x_1\le x_2\le x_3,
a_1a_2a_3+\delta\le x_1x_2x_3\le b_1b_2b_3-\delta\right\}
\end{eqnarray*}
$\displaystyle{
c(\delta,\alpha)=\min\left\{\delta,\sqrt{\frac{\alpha_1\alpha_2\alpha_3}{\alpha_1+\delta}}-
\frac{\alpha_1\alpha_2}{\alpha_1+\delta}\right\}}$, $\alpha=(\alpha_1,\alpha_2,\alpha_3)$ and
$$ \begin{array}{l}\displaystyle
\Lambda_{K_\delta}:=\bigcup_{(\alpha_1,\alpha_2,\alpha_3)\in\Lambda_{\widetilde{E}_\delta}}
\left\{(x_1,x_2,x_3)\in\mathbb{R}^3:\ x_1x_2x_3=\alpha_1\alpha_2\alpha_3,\right.\vspace{0.1cm}\\ \left.
\hspace{5.6cm} 0<x_1\le x_2\le x_3\le\alpha_3,\ x_2x_3\le\alpha_2\alpha_3\right\}.\end{array}$$
\end{proof}

\section*{Acknowledgements}
The research of Ana Cristina Barroso was partially supported
by Funda\c c\~ao para a Ci\^{e}ncia e Tecnologia, through Financiamento Base 2010-ISFL/1/209,
PTDC/MAT/109973/2009 and\, UTA-CMU/MAT/0005/2009.
The research of Ana Margarida Ribeiro was partially supported
by Funda\c c\~ao para a Ci\^{e}ncia e Tecnologia, through Financiamento Base 2010-ISFL\-/1/297,
\-PTDC/MAT/\-109973/\-2009 and\, UTA-CMU/MAT/0005/2009.
This work was undertaken during a visit of Ana Margarida Ribeiro
to the Mathematics Department of the \'Ecole Polytechnique F\'ed\'erale de Lausanne
whose kind hospitality and support is gratefully acknowledged.


\begin{thebibliography}{99}

\bibitem{BCR}A.C.~Barroso, G.~Croce and A.M.~Ribeiro.
\newblock Existence for a nonlinear problem involving isotropic deformations.
\newblock {\em Journal of Nonlinear Systems and Applications}, 1:113-121, 2010.

\bibitem{CT}P.~Cardaliaguet and R.~Tahraoui.
\newblock Equivalence between rank one convexity and polyconvexity for
isotropic sets of $\mathbb R\sp {2\times 2}$ (Part I).
\newblock {\em Nonlinear Anal.}, 50:1179-1199, 2002.

\bibitem{CT2}P.~Cardaliaguet and R.~Tahraoui.
\newblock Equivalence between rank one convexity and polyconvexity for
isotropic sets of $\mathbb R\sp {2\times 2}$ (Part II).
\newblock {\em Nonlinear Anal.}, 50:1201-1239, 2002.

\bibitem{CDLMR}S.~Conti, C.~De Lellis, S.~M\"uller and M.~Romeo.
\newblock Polyconvexity equals rank-one convexity for connected isotropic
sets in $\mathbb M\sp {2\times 2}$.
\newblock {\em C. R. Acad. Sci. Paris S\'er. I Math.}, 337(4):233--238, 2003.

\bibitem{Crocethesis}G.~Croce.
\newblock Sur quelques inclusions diff\'erentielles de premier et de
deuxi\`eme ordre.
\newblock {\em Ph.D. Thesis No. 3010, \'Ecole Polytechnique F\'ed\'erale de Lausanne}, 2004.

\bibitem{Croce}G.~Croce.
\newblock A differential inclusion: the case of an isotropic set.
\newblock {\em ESAIM Control Optim. Calc. Var.}, 11:122-138, 2005.

\bibitem{DirectMethods}
B.~Dacorogna.
\newblock {\em Direct methods in the calculus of variations, 2nd ed.}, volume~78 of {\em
  Applied Mathematical Sciences}.
\newblock Springer-Verlag, Berlin, 2008.

\bibitem{Dac-Marc-CR}B.~Dacorogna and P.~Marcellini.
\newblock Th\'eor\`emes d'existence dans les cas scalaire et vectoriel pour les \'equations de {H}amilton-{J}acobi.
\newblock {\em C. R. Acad. Sci. Paris S\'er. I Math.}, 322(3):237--240, 1996.

\bibitem{Dac-Marc}B.~Dacorogna and P.~Marcellini.
\newblock {\em Implicit partial differential equations}.
\newblock Progress in Nonlinear Differential Equations and their Applications,
  37. Birkh\"auser Boston Inc., Boston, 1999.

\bibitem{Dacorogna-Pisante}B.~Dacorogna and G.~Pisante.
\newblock A general existence theorem for differential inclusions in the vector valued case.
\newblock {\em Port. Math. (N.S.)}, 62(4):421--436, 2005.

\bibitem{Dacorogna-Ribeiro}B.~Dacorogna and A.M.~Ribeiro.
\newblock On some definitions and properties
of generalized convex sets arising in the calculus of variations.
\newblock
{\em Recent Advances on Elliptic
and Parabolic Issues, Proceedings of the 2004 Swiss-Japanese Seminar}, World Scientific Publishing, 103-128, 2006.

\bibitem{Dacorogna-Tanteri}B.~Dacorogna and C.~Tanteri.
\newblock Implicit partial differential equations and the constraints of
  nonlinear elasticity.
\newblock {\em J. Math. Pures Appl.}, 81(4):311--341, 2002.

\bibitem{HJ}R.A.~Horn and Ch.R.~Johnson.
\newblock {\em Topics in matrix analysis.}
\newblock Cambridge University Press, Cambridge, 1991.

\bibitem{Kirchheim}
B.~Kirchheim.
\newblock Deformations with finitely many gradients and stability of
  quasiconvex hulls.
\newblock {\em C. R. Acad. Sci. Paris S\'er. I Math.}, 332(3):289--294, 2001.

\bibitem{Kirchheim-notes}
B.~Kirchheim.
\newblock Rigidity and geometry of microstructures.
\newblock {\em Lecture Note No. 16, Max Planck Institute for Mathematics in the
  Sciences}, 2003.

\bibitem{Kolar}
J.~Kol{\'a}{\v{r}}.
\newblock Non-compact lamination convex hulls
\newblock {\em Ann. Inst. H. Poincar\'e Anal. Non Lin\'eaire}, 20(3):391-403, 2003.

\bibitem{Muller-Sverak1996}
S.~M{\"u}ller and V.~{\v{S}}ver{\'a}k.
\newblock Attainment results for the two-well problem by convex integration
\newblock {\em Geometric analysis and the calculus of variations, Internat. Press, Cambridge, MA}:
239--251, 1996.

\bibitem{Muller-Sverakcounterexamples}
S.~M{\"u}ller and V.~{\v{S}}ver{\'a}k.
\newblock Convex integration for {L}ipschitz mappings and counterexamples to
  regularity.
\newblock {\em Ann. of Math. (2)}, 157(3):715--742, 2003.

\bibitem{Muller-Sychev} S.~M{\"u}ller and M.~A. Sychev.
\newblock Optimal existence theorems for nonhomogeneous differential
  inclusions.
\newblock {\em J. Funct. Anal.}, 181(2):447--475, 2001.

\bibitem{Pedregal} P.~Pedregal.
\newblock Laminates and microstructure.
\newblock {\em European J. Appl. Math.}, 4(2):121--149, 1993.

\bibitem{Ribeiro-thesis} A.M.~Ribeiro.
\newblock Inclusions diff\'erentielles et probl\`emes variationnels.
\newblock {\em Ph.D. Thesis No. 3583, \'Ecole Polytechnique F\'ed\'erale de Lausanne}, 2006.

\end{thebibliography}
\end{document}